\documentclass[11pt]{amsart}

\usepackage{enumerate}
\usepackage{remreset}
\usepackage[ansinew]{inputenc}
\usepackage[english]{babel}

\usepackage{amsmath}        %
\usepackage{amsthm}         %
\usepackage{amscd}          %
\usepackage{amsfonts}       
\usepackage{amssymb}        
\usepackage{listings}
\usepackage{upgreek}

\usepackage{tikz}
\usepackage{graphicx}

\theoremstyle{definition}

\theoremstyle{plain}
\newtheorem{theorem}{Theorem}[section]

\newtheorem{lem}[theorem]{Lemma}
\newtheorem{prop}[theorem]{Proposition}

\newtheorem*{theorem*}{Theorem}
\newtheorem*{theoremA}{Theorem A}
\newtheorem*{theoremB}{Theorem B}
\theoremstyle{remark}

\makeatletter
\def\imod#1{\allowbreak\mkern10mu({\operator@font mod}\,\,#1)}
\makeatother

\newcommand{\eqnref}[1]{(\ref{#1})}

\newcommand{\HR}{\frac{1}{i}\mathbb{H}}
\newcommand{\C}{\mathbb{C}}
\newcommand{\R}{\mathbb{R}}
\newcommand{\Z}{\mathbb{Z}}
\newcommand{\N}{\mathbb{N}}
\newcommand{\Q}{\mathbb{Q}}

\newcommand{\sgn}{\mathrm{sgn}}

\renewcommand{\Re}[1]{\mathrm{Re}\left( #1\right) }
\renewcommand{\Im}[1]{\mathrm{Im}\left( #1\right) }

\newcommand{\rhoT}{\rho}
\newcommand{\gammaGCD}{\sigma}
\newcommand{\gammaCo}{\upgamma}
\newcommand{\kappaH}{\kappa^\ast}

\begin{document}
\title[Asymtotics for moments of higher ranks]{Asymptotics for
moments of higher ranks}
\author{Matthias Waldherr}
\thanks{The author is supported by Graduiertenkolleg "Global Structures in Geometry und Analysis"}

\address{Mathematical Institute\\University of
Cologne\\ Weyertal 86-90 \\ 50931 Cologne \\Germany}
\email{mwaldher@math.uni-koeln.de}

\begin{abstract}
Bringmann, Mahlburg, and Rhoades have found asymptotic expressions
for all moments of the partition statistics rank and crank. In this
work we extend their methods to higher ranks. The $T$-rank,
introduced by Garvan, for odd integers $T=3$ is a natural
generalization of the rank ($T=3$) and crank ($T=1$).
\end{abstract}

\maketitle


\section{Introduction and Statement of Results}
In the theory of integer partitions the partition statistics "rank"
(defined by Dyson) and "crank"(defined by Andrews and Garvan) play a
fundamental role.

The rank was first introduced by Dyson \cite{dyson} in the attempt
to explain the Ramanujan congruences for the partition function from
a combinatorial point of view.
\begin{align*}
p(5n+4) &\equiv 0 \imod{5},\\
p(7n+5) &\equiv 0 \imod{7},\\
p(11n+6) &\equiv 0 \imod{11}.
\end{align*}
Indeed, Atkin and Swinnerton-Dyer \cite{AtSD} proved that the rank
explains the first two congruences. Dyson already
observed that the third congruence could not be explained by the
rank, which led him to speculate about the existence of a different
partition statistic, which he termed the "crank", which would
explain all three congruences. This statistic was later found by
Andrews and Garvan \cite{andgarv}.

While the combinatorial definitions of the rank and the crank are
very different and do not allow for an immediate generalization,
Garvan \cite{Ga_gendy} found a generalization by looking at
generating functions. For an odd positive integer $T$, he defined
the numbers $N_T(m,n)$ by the following series
\[ \sum_{n=0}^\infty N_T(m,n)q^n=\frac{1}{(q)_\infty} \sum_{n=1}^\infty (-1)^{n-1}q^{\frac{n}{2}(Tn-1)+|m|n}(1-q^n), \]
where the $q$-Pochhammer symbol is defined by
$(x;q)_n:=\prod_{k=1}^n (1-xq^k)$. Then, $N_1(m,n)$ is the number of
partitions of $n$ having crank $m$ and $N_3(m,n)$ is the number of
partitions of $n$ having rank $m$. Although $N_T$ is not the
counting function for a partition statistic in a strict sense,
Garvan \cite{Ga_gendy} also found combinatorial interpretations of
the numbers $N_T(m,n)$.

The rank and crank received renewed interest only recently when Atkin and
Garvan \cite{AtGar} defined rank and crank moments and Andrews
\cite{An_spt} discovered a combinatorial meaning of these moments.

We now define moments not only for the rank and crank but also for
the $T$-ranks introduced by Garvan. For a positive odd integer $T$
and a positive integer $r$ we define
\[ m^r_T(n):=\sum_{m=-n}^n m^r N_T(m,n).\]

Andrews \cite{An_spt} related his smallest parts function to rank and crank moments. More precisely, if we let
$\mathrm{spt}(n)$ denote the total number of smallest parts counted
in all partitions of $n$, then
\[ \mathrm{spt}(n)=\frac{1}{2}\left(m_1^2(n)-m_3^2(n)\right).\]

In \cite{ga_cong} Garvan raised the question whether
$m_1^r(n)>m_3^r(n)$ holds for all positive even integers $r$. In
\cite{BrMa_ineq}, Bringmann and Mahlburg realized that one could
prove Garvan's conjecture (for $n$ large enough), if one knew
precise asymptotics of $m_1^r(n)$ and $m_3^r(n)$. Indeed, Bringmann
and Mahlburg found such asymptotic expressions for $m_1^r(n)$ and
$m_3^r(n)$ and were able to prove that Garvan's conjecture holds in
the cases $r=2,4$ for sufficiently large $n$ large. After that,
Bringmann, Mahlburg, and Rhoades in \cite{BrMaRh_asym} succeeded in
proving asymptotics for all even $r$, thus settling Garvan's
conjecture in the limit case. In \cite{ga_high}, Garvan proved his
conjecture for all $n$ and $r$ by finding a combinatorial
interpretation of the difference of the rank and crank moments in
terms of higher $\mathrm{spt}$-functions.

In this work, we will address the same questions for the moments of
higher rank functions $N_T$. This means that our first goal is to
derive asymptotic formulas for the moments of the $T$-ranks. Our second objective is
to prove an analogue of Garvan's conjecture.

This research was part of the authors PhD thesis \cite{Wal_thesis}.
At first, this was not motivated mainly by combinatorics but rather aimed
at developing further the methods in the context of automorphic
forms, which are used to derive the asymptotic formulas. However,
after this work was finished, also a combinatorial interpretation of
moments of higher ranks was found by Dixit and Yee \cite{DixitYee}.
Their interpretation gives Garvan's conjecture in the general case.
Furthermore, our main theorem can be used to find asymptotic formulas for their
"generalized higher $\mathrm{spt}$-functions".

We now briefly outline our approach behind finding the asymptotic
formulas for $m_T^r(n)$. The general idea is to apply the Circle
Method to the generating function
\[ M^r_T(q):= \sum_{n=0}^\infty m^r_T(n) q^n= \sum_{n=0}^\infty \sum_{m=-n}^n m^r N_T(m,n) q^n.\]
The philosophy of the Circle Method is as follows. The generating
function $M^r_T(q)$ defines a holomorphic function on the unit
circle $|q|<1$ with singularities at the boundary. If one is able to
find nice enough expressions for the shape of these singularities in
the neighborhood of all roots of unity, one can use Cauchy's
integral formula to determine asymptotic expressions for $m_T^r(n)$.

The proof in \cite{BrMaRh_asym} relies on the Circle Method combined
with complicated recurrence relations for the rank and the crank. As
a consequence, this approach can not be applied to the case of
higher ranks. However, in \cite{BrMaRh_stat}, Bringmann, Mahlburg,
and Rhoades, using a new approach, improved upon their previous work
and found asymptotic formulas for all rank and crank moments with
error terms which are as small as one could hope for by using the
Circle Method. The idea behind their new approach is the insight
that it is easier to not study the generating functions $M^r_T(q)$
individually, but rather a two-variable generating function
involving $M^r_T(q)$ for all $r$ at the same time. In fact, it turns
out that there is a two-variable function $\mathcal{M}_T(u,q)$ whose
Taylor coefficients turn out to be $M^r_T(q)$. Moreover, this
generating function is a Jacobi form in the crank case (i.e. $T=1$)
and a mock Jacobi form in the rank case (i.e. $T=3$). Now the idea
is simple: One determines explicit transformation formulas for the
(mock) Jacobi forms, uses an asymptotic Taylor expansion to obtain
asymptotic formulas for $M^r_T(q)$ for all $r$ in the neighborhood
of roots of unity and then applies the Circle Method.

The generalization of this idea to the higher ranks case, however,
is not straightforward but some complications occur. Indeed, for
$T=1$ the function $\mathcal{M}_1(u,q)$ is easily seen to be a
Jacobi form. For $T=3$, we only obtain a mock Jacobi form
\footnote{We do not give formal definitions for Jacobi forms or mock
Jacobi forms, as we will need no results from a general theory, but
only require the transformation laws. For an account on the theory
of Jacobi forms we refer to \cite{zag_jac}. Mock Jacobi forms are a
relatively new object of study. For further information the reader
may consult \cite{BrRi_mock},  \cite{zwe_thesis}, and
\cite{zag_bourbaki}.}. As a consequence the determination of the
explicit transformation laws becomes a rather technical issue. For
$T>3$ we have to deal with even more complicated expressions.
Luckily, this turns out to be challenging only from a notational
point of view. However, a more serious issue arises for $T>3$. While
for $T=1$ we obtain a Jacobi form, for $T>1$ we no longer have true
but only mock Jacobi transformation laws. The obstruction to
modularity in the case $T=3$, however, turns out to not affect the
asymptotic behavior of $M^r_T(q)$ at roots of unity significantly.
For $T>3$ this is no longer true and we have to take into account
also the contribution from the obstruction to modularity.

To state our main theorems, we need to introduce some notation. For
positive integers $a,b,c$, we define the constants $\kappa(a,b,c)$
and $\kappaH(a,b,c)$ as in equation
\eqnref{garvan_eqn_genasym_kappa} and
\eqnref{garvan_eqn_genasym_kappaH} on page
\pageref{garvan_eqn_genasym_kappaH}. Furthermore, we let $K_k(n)$
and $K_{\sigma,\rho,l;k}(n)$ denote the Kloosterman sums in equation
\eqnref{garvan_eqn_circle_klooster1} on page
\pageref{garvan_eqn_circle_klooster1} and in equation
\eqnref{garvan_eqn_circle_klooster2} on page
\pageref{garvan_eqn_circle_klooster2}. We write $I$ for the modified
Bessel function and define $\mathcal{I}$ to be an integral over a
modified Bessel function as in equation
\eqnref{garvan_eqn_integral_besselint} on page
\pageref{garvan_eqn_integral_besselint}. Finally, the notation $( \
\ \ )_+$ means that we only include the expression if the value in
the parenthesis is bigger than 0, and regard it as 0 else. Now our
main theorem reads as follows:
\begin{theoremA}
Let $T<24$ be an odd integer and $r$ an even integer. Then, we have
\begin{align*}
&m_T^r(n)= 2 \pi \sum_{k \leq \sqrt{n}} \frac{K_k(n)}{k}
\sum_{2a+2b+2c=r} \kappa(a,b,c) (kT)^a
(24n-1)^{-\frac{3}{4}+\frac{a}{2}+c}
I_{-\frac{3}{2}+a+2c}\left(\tfrac{\pi \sqrt{24n-1}}{6k}\right)\\
&+2 \pi \sum_{\gammaGCD|T} \sum_{\substack{t=-\frac{T-1}{2} \\ t
\neq 0}}^{\frac{T-1}{2}}
\sum_{\varrho=-\frac{T-1}{2}}^{\frac{T-1}{2}}
 \sum_{\substack{0<k \leq \sqrt{n}\\(k,T)=\gammaGCD}}
\sum_{l=0}^{\frac{k}
{\gammaGCD}-1} \frac{K_{\sigma,\varrho,l;k}(n)}{k} \sum_{2a+(2b+1)+c=r} \hspace*{-1.5em} \kappaH(a,b,c) k^{b-\frac{1}{2}} T^{b-\frac{1}{2}} \gammaGCD^{c+\frac{1}{2}} \left(2n-\tfrac{1}{12}\right)^{\frac{a+c}{2}-\frac{1}{4}}\\
&\times\!\!\left(\tfrac{1}{12}-\tfrac{\gammaGCD^2}{T^3}\left(\varrho^2+\tfrac{T^2}{4}-|\varrho|T\right)\right)_+^{\frac{3}{4}-\frac{a+c}{2}}
\!\mathcal{I}_{T;\alpha_{T,t}\left(l,\frac{k}{\gammaGCD}\right),\frac{1}{12}-\frac{\gammaGCD^2}{T^3}\left(\varrho^2+\frac{T^2}{4}-|\varrho|T\right),
-\frac{1}{12},\frac{\varrho}{T}}\left(c,-\tfrac{1}{2}-a-c,k;n\right)\!+\!
E_{r,T}(n).
\end{align*}
Here, $E_{r,T}(n)$ is an error term. If $r=2$, then the error term
has the magnitude $O_{T}\left(n \log n\right)$, whereas it is of
order $O_{r,T}\left(n^{r-1}\right)$ if $r>2$.
\end{theoremA}

Theorem A is a direct generalization of \cite{BrMaRh_stat}. Their
theorem states that
\[ m_T^r(n)= 2 \pi \sum_{k \leq \sqrt{n}} \frac{K_k(n)}{k}
\sum_{2a+2b+2c=r} \kappa(a,b,c) (kT)^a
(24n-1)^{-\frac{3}{4}+\frac{a}{2}+c}
I_{-\frac{3}{2}+a+2c}\left(\tfrac{\pi \sqrt{24n-1}}{6k}\right) +
O\left(n^{r-1}\right) \] for $T=1$ and $T=3$. Indeed, the additional
summand in Theorem A will only occur for $T>3$ and comes from the
fact that we cannot neglect the term coming from the obstruction to
modularity.

In the statement of Theorem A we observe the occurrence of modified
Bessel functions with different indices. This accounts to the fact
that the moment generating functions $M_T^r(q)$ exhibit a
transformation behavior with mixed weights (in fact the weights
range from $-\frac{1}{2}$ to $r-\frac{1}{2}$). Indeed, the error
terms of Theorem A are best possible which one can obtain with the
Circle Method. This occurs because the Circle Method cannot detect
holomorphic modular forms of positive weight. The size of the error
terms in Theorem A are just as big as the coefficients of
holomorphic Eisenstein series of the weights occurring in the
transformation behavior of the moment generating functions.

From Theorem A we deduce our second main theorem and an analogue of
Garvan's conjecture.

\begin{theoremB}
Let $T<24$ be an odd integer and $r$ an even integer. Then, as $n
\to \infty$, we have
\[ m_T^r(n) \sim 2\sqrt{3}(-1)^\frac{r}{2}B_{r}\left(\tfrac{1}{2}\right)(24n)^{\frac{r}{2}-1}e^{\pi
\sqrt{\frac{2n}{3}}},\] where $B_{r}\left(\cdot\right)$ is a
Bernoulli polynomial. Furthermore,
\[ m_{T-2}^r(n)-m_{T}^r(n) \sim \sqrt{3}\frac{r!}{(r-2)!} (-1)^{\frac{r}{2}+1}  B_{r-2}\left(\tfrac{1}{2}\right)(24n)^{\frac{r}{2}-\frac{3}{2}}e^{\pi
\sqrt{\frac{2n}{3}}}.\] In particular, $m_{T-2}^r(n)>m_T^r(n)$ for
all sufficiently large $n$.
\end{theoremB}

\section{Preliminaries and Notation}
We first recall transformation formulas of the $\eta$- and $\theta$-function and provide transformation laws for Appell-Lerch sums and their
completions found by Zwegers \cite{zwe_thesis}. Furthermore, we set up the notation for the rest
of this paper.
\subsection{Transformation laws}
Throughout this work, $z$ denotes a complex number satisfying $\Re{z}>0$ and we set $q:=e^{-2 \pi z}$.

We define the Dedekind $\eta$-function by
\[ \eta(iz):= q^\frac{1}{24} \prod_{n=1}^\infty (1-q^n).\]
If $h,k$ are coprime positive integers then
\[ \eta\left( \tfrac{1}{k}(h+iz)\right)=\sqrt{\tfrac{i}{z}}\chi\left(h,[-h]_k,k\right) \eta\left( \tfrac{1}{k}\left([-h]_k+\tfrac{i}{z}\right)\right),\]
where
\[ \chi\left(h,[-h]_k,k\right) :=\left\{
\begin{array}{ll}
\left( \frac{h}{k}\right)e^{-\frac{ \pi i k}{4}} e^{\frac{\pi i }{12}\left( -\beta[-h]_k\left(1-k^2\right)+k\left(h-[-h]_k\right)\right)} & \hbox{if } k \text{ is odd}, \\
e^{-\frac{\pi i }{4}} \left( \frac{k}{h}\right) e^{\frac{\pi i
}{12}\left(
hk\left(1-[-h]_k^2\right)-[-h]_k\left(\beta-k+3\right)\right)}&
\hbox{if } h \text{ is odd}.
                        \end{array}
                      \right.
 \]
Here and in the following, $[\cdot]_k$ denotes the inverse modulo $k$ and $\beta$ is
defined by $-h[-h]_k-\beta k=1$.

Next, we consider the classical Jacobi theta function, which, for $v
\in \C$, is defined by
\[ \theta(v;iz):=\sum_{\nu \in \frac{1}{2}+\Z} e^{-\pi \nu^2z+2\pi i \nu\left(v+\frac{1}{2}\right)}.\]
For $\theta(v;iz)$ there is a well-known product expansion:
\begin{equation}\label{pre_eqn_pre_thetaproduct}\theta(v;iz)=-ie^{-\frac{\pi z}{4}} e^{-\pi i v}\prod_{n=1}^\infty \left(1-e^{-2 \pi n z }\right)\left(1-e^{2 \pi i v}e^{-2 \pi (n-1)z}\right)\left(1-e^{-2 \pi i v}e^{-2 \pi
n z}\right).\end{equation}

The Jacobi theta function satisfies elliptic and modular
transformation properties. To state these precisely, let $h,k$ be
coprime integers. Then, for any $n \in \Z$, we have
\[\theta(v+1;iz)=-\theta(v;iz),  \quad \theta(-v;iz)=-\theta(v;iz), \quad \theta(v+niz;iz)=(-1)^ne^{\pi n^2 z-2\pi i n v}\theta(v;iz).\] Furthermore, we have
\[ \theta\left(v;\tfrac{1}{k}(h+iz)\right)=\sqrt{\tfrac{i}{z}}\chi^3(h,[-h]_k,k) e^{-\frac{\pi k
v^2}{z}}\theta\left(\tfrac{iv}{z};\tfrac{1}{k}\left([-h]_k+\tfrac{i}{z}\right)\right).
\]

Following Zwegers \cite{zwe_thesis}, for $u,v\in \C$ and
$u \notin \Z+iz\Z$, we define the function
\begin{equation}\label{pre_eqn_pre_zwegersA}
\mathrm{A}(u,v;iz):=e^{\pi i u}\sum_{n \in \Z} \frac{(-1)^n e^{-\pi
(n^2+n)z+2 \pi i nv}}{1-e^{-2 \pi nz +2 \pi i u}}.
\end{equation}
and for $u,v \in \C \setminus
(\Z+iz\Z)$, we define
\begin{equation}\label{pre_eqn_pre_zwegersmu}
\mu(u,v;iz):=\frac{\mathrm{A}(u,v;iz)}{\theta(v;iz)}=\frac{e^{\pi i
u}}{\theta(v;iz)} \sum_{n \in \Z} \frac{(-1)^n e^{-\pi (n^2+n)z+2
\pi i nv}}{1-e^{-2 \pi nz +2 \pi i u}}.
\end{equation}

The $\mu$-function itself does not transform as a modular form.
Zwegers discovered that one can complete $\mu$ to a function
$\widehat{\mu}$ having nice transformation properties. To define
this completion requires the following non-holomorphic function:
\begin{equation}\label{pre_eqn_pre_zwegersR} R(w;iz):= \sum_{\nu
\in \Z + \frac{1}{2}} (-1)^{\nu-\frac{1}{2}} \left( \sgn(\nu) -
E\left(\sqrt{2 \Im{iz}}\left(\nu + \tfrac{\Im{w}}{
\Im{iz}}\right)\right)\right) e^{\pi \nu^2 z} e^{-2 \pi i \nu w},
\end{equation} where $E(z) := 2 \int_0^z e^{-\pi u^2} du$.
Following Zwegers, we define
\begin{equation}\label{pre_eqn_pre_zwegersmuhat}\widehat{\mu}(u,v;iz):=\mu(u,v;iz)+\tfrac{i}{2}R(u-v;iz).
\end{equation}
\begin{lem}[\cite{zwe_thesis}  Theorem 1.11]\label{pre_fact_pre_muhat}
Let $u,v \in \C \setminus (\Z+iz\Z)$. If $h,k$ are coprime integers,
and $m,m',n,n' \in \Z$, then we have:
\begin{enumerate}[a)]
\item $\widehat{\mu}(u+miz+n,v+m'iz+n';iz)=(-1)^{m+n+m'+n'}e^{-\pi z(m-m')^2+2\pi
i(m-m')(u-v)}\widehat{\mu}(u,v;iz)$,
\item $\widehat{\mu}(-iuz,-ivz; \frac{1}{k}(h+iz))=\chi^{-3}\left(h,[-h]_k,k\right) \sqrt{\frac{i}{z}} e^{-\pi kz(u-v)^2}
\widehat{\mu}\left(u,v;\frac{1}{k}\left([-h]_k+\tfrac{i}{z}\right)\right)$.
\end{enumerate}
\end{lem}

The function $R$ itself also satisfies some transformation formulas.
In order to be able to state these, we introduce the Mordell
integral for $w \in \C$ as
\begin{equation}\label{pre_eqn_pre_zwegersH}
H(w;z):=\int_\R \frac{e^{-\pi i x^2 z-2 \pi wx}}{\cosh \pi x} dx.
\end{equation}

\begin{lem}[\cite{zwe_thesis} Propositions 1.9 and 1.10]\label{pre_fact_pre_R}
For any $w \in \C$ we have
\begin{enumerate}[a)]
\item $R(w+1;iz)=-R(w;iz)$,
\item $R(w;iz+1)=e^{-\frac{\pi i }{4}}R(w;iz)$,
\item $R(w;iz)=-\frac{1}{\sqrt{z}}e^\frac{\pi w^2}{z} \left(
R\left(\tfrac{iw}{z};\tfrac{i}{z}\right)-H\left(\tfrac{iw}{z};\tfrac{i}{z}\right)\right)$.
\end{enumerate}
Furthermore, in \cite{BrFo_kac} the following dissection property is
proved:
\begin{enumerate}[d)]
\item $R\left(w; \tfrac{iz}{n}\right)=\sum_{l=0}^{n-1} e^{\frac{\pi}{n}\left(l-\frac{n-1}{2}\right)^2z}e^{-2\pi i\left(l-\frac{n-1}{2}\right)\left(w+\frac{1}{2}\right)} R\left(nw + (l-\tfrac{n-1}{2})iz+\tfrac{n-1}{2}; n
iz\right).$
\end{enumerate}
\end{lem}

Finally, in recent unpublished work \cite{zwe_multi}, for $u,v \in \C$ and $u \notin iz\Z+\Z$ Zwegers
introduced a new function $\mathrm{A}_T$ by
\begin{equation}\label{pre_eqn_pre_zwegershigherAp} \mathrm{A}_T(u,v;iz):= e^{\pi i u T} \sum_{n \in \Z} \frac{(-1)^{Tn}e^{-\pi T n(n+1)z}e^{2 \pi i n v}}{1-e^{2\pi i u}e^{-2 \pi z}}.
\end{equation}
As observed by Zwegers, this function is related to the functions
discussed before in the following way:
\begin{lem}\label{pre_fact_pre_AT}
The function $A_T$ satisfies the following properties:
\begin{enumerate} [a)]
\item $\mathrm{A}_T(u,v;iz)= \sum_{t=0}^{T-1} e^{2 \pi i u t}
\mathrm{A}_1\left(Tu,v+tiz+\tfrac{T-1}{2}; Tiz\right)$.
\item  $\mathrm{A}_1(u,v;iz)=\theta(v;iz)\mu(u,v;iz)=\mathrm{A}(u,v;iz)$.
\end{enumerate}
\end{lem}

\subsection{Notation}
We will now introduce the notation that we use for the rest of this paper. Since we (eventually) want to apply the Circle Method, we
will be interested in transformations where we replace $q=e^{-2 \pi
z}$ by $e^{\frac{2 \pi i}{k}(h+iz)}$, for $z \in \C$ with
$\Re{z}>0$, for $k$ being a positive integer, and $0 \leq h < k$
with $h$ coprime to $k$. We will also assume that $|z|<1$. Instead
of repeating these conditions, we will briefly write that ``$h,k$
and $z$ satisfy the usual conditions''.

For a fixed odd integer $T$, we write $\gammaGCD:=(T,k)$ and
$\gammaCo:=\frac{T}{(T,k)}$. 
We denote by $\rhoT$ the function $\rhoT:\Z \to
\left\{-\frac{T-1}{2},\ldots,\frac{T-1}{2}\right\}$ which assigns to
every integer its residue class modulo $T$ with smallest absolute
value.

In the course of proving transformation formulas, we will encounter
several expressions which are roots of unity, and which occur as a
result of the roots of unity appearing in the transformation laws of
$\eta,\theta$ and $\widehat{\mu}$. For odd integers $T$,
$-\tfrac{T-1}{2} \leq t \leq \tfrac{T-1}{2}$ and $h,k$ coprime, we
set
\begin{equation*}\label{garvan_eqn_trans_utheta} U_\theta(T,t,h,k):=(-1)^{\frac{(th\gammaCo -\rhoT(t\gammaCo
h))}{T}}e^{\frac{\pi i(t\gammaCo h-\rhoT(t\gammaCo h))^2}{\gammaCo
Tk}[-\gammaCo h]_{\frac{k}{\gammaGCD}}} e^{-\frac{2\pi i t
\rhoT(t\gammaCo h)}{\gammaCo Tk}},
\end{equation*}
\begin{equation*}\label{garvan_eqn_trans_Umu}
U_\mu(T,t,h,k):=\chi^{-3}\left(h,[-h]_k,k\right)(-1)^{th-\rhoT(th)}
e^{-\frac{\pi i[-h]_k}{k}\left(\frac{th-\rhoT(th)}{T}\right)^2}
e^{\frac{2\pi i t\rhoT(th)}{T^2k}},
\end{equation*}
and define $U^{\ast}_\theta(T,t,h,k)$\label{garvan_eqn_genasym_uthetastar} by
\[\left\{%
\begin{array}{ll}
    -2\chi^3\left(\gammaCo h,[-\gammaCo h]_{\frac{k}{\gammaGCD}},\tfrac{k}{\gammaGCD}\right)
U_\theta(T,t,h,k)e^{\frac{ \pi i}{4k} \gammaGCD[-\gammaCo
h]_{\frac{k}{\gammaGCD}}} \sin \left( -\tfrac{\pi t}{\gammaCo
k}\left(1+\gammaCo h[-\gammaCo h]_{\frac{k}{\gammaGCD}}\right)
\right) & \hbox{if } \rhoT(\gammaCo h t)=0,\\
    -i\chi^3\left(\gammaCo h,[-\gammaCo h]_{\frac{k}{\gammaGCD}},\tfrac{k}{\gammaGCD}\right)
U_\theta(T,t,h,k)e^{\pi i\left(\frac{\gammaGCD}{4k}([-\gammaCo
h]_{\frac{k}{\gammaGCD}}-\left( \frac{\rhoT(t\gammaCo h)}{\gammaCo
k}[-\gammaCo h]_{\frac{k}{\gammaGCD}} -\frac{t}{\gammaCo
k}\left(1+\gammaCo h[-\gammaCo h]_{\frac{k}{\gammaGCD}}\right)
\right)\right)} & \hbox{if } \rhoT(\gammaCo h t)>0,\\
i\chi^3\left(\gammaCo h,[-\gammaCo
h]_{\frac{k}{\gammaGCD}},\tfrac{k}{\gammaGCD}\right)
U_\theta(T,t,h,k)e^{\pi i\left(\frac{\gammaGCD}{4k}([-\gammaCo
h]_{\frac{k}{\gammaGCD}}+\left( \frac{\rhoT(t\gammaCo h)}{\gammaCo
k}[-\gammaCo h]_{\frac{k}{\gammaGCD}} -\frac{t}{\gammaCo
k}\left(1+\gammaCo h[-\gammaCo h]_{\frac{k}{\gammaGCD}}\right)
\right) \right)} & \hbox{if } \rhoT(\gammaCo h t)<0. \\
\end{array}%
\right.    \]

Furthermore, for $0 \leq l \leq k-1$ we set
\[ \alpha_{T,t}(l,k):=\tfrac{1}{k}\left(-\tfrac{t}{T}+\left(l-\tfrac{k-1}{2}\right)\right).\]
Note that $\left|\alpha_{T,t}(l,k)\right|< \frac{1}{2}$. Moreover, we define
\begin{equation*}\label{garvan_eqn_trans_UH}
U_H(T,t,l,h,k):=e^{-\frac{\pi i(hk+1)}{4}} (-1)^{lh +
\frac{(k-1)(h-1)}{2} +th-\rhoT(th)+1}e^{-\frac{\pi
ih}{k}\left(l-\frac{k-1}{2}\right)^2}e^{-2\pi
i\left(\left(l-\frac{k-1}{2}\right)\left(\frac{1}{2}-\frac{th}{Tk}\right)+
\frac{\rhoT(th)}{T}\alpha_{T,t}(l,k)\right)},
\end{equation*}
and
\begin{equation}\label{garvan_eqn_moment_uhstar} U_H^\ast(T,t,l,h,k):= i^\frac{3}{2}\frac{U_\theta^\ast(T,t,h,k)}{\chi\left(h,[-h]_k,k\right)}U_H\left(T,t,l,\gammaCo h,\tfrac{k\gammaCo }{T}\right)e^{2\pi i
\frac{\rhoT(t\gammaCo h)}{T} \alpha_{T,t}\left(l,\frac{k\gammaCo
}{T}\right) } e^{\frac{\pi i h}{12k}}e^{\frac{\pi i [-h]_k}{12k}},
\end{equation}

Finally, we require the following constants
\begin{equation}\label{garvan_eqn_genasym_kappa} \kappa(a,b,c):= \frac{(2(a+b+c))!}{a!(2b+1)!(2c)!} \frac{(-1)^{a+c}}{\pi^a2^{2(a+b)}}  B_{2c}\left(\tfrac{1}{2}\right),
\end{equation}
for $a,b,c \in \N_0$ and $\kappa(a,b,c)=0$ otherwise. Here $B_{2c}$
denotes the $2c$-th Bernoulli polynomial.
\begin{equation}\label{garvan_eqn_genasym_kappaH} \kappaH(a,b,c):=\frac{(2a+(2b+1)+c)!}{a!(2b+1)!c!} \frac{(-1)^{a+c+1} }{\pi^a 2^{2a+2b+1}}
\end{equation}
for $a,b,c \in \N_0$ and $\kappaH(a,b,c)=0$ otherwise.

One easily shows that these numbers appear as Taylor coefficients in the Taylor expansions
\[ e^{\frac{\pi \nu u^2}{z}} \frac{\sin(\pi u)}{\sinh(\frac{\pi u }{z})} = \sum_{r=0}^\infty \sum_{2a+2b+2c=r} \kappa(a,b,c) \nu^a z^{1-a-2c} \tfrac{(2 \pi i u)^r}{r!}\]
and
\[ \sin(\pi u) e^{\frac{\pi \nu u^2}{z}} e^{-\tfrac{2 \pi i \lambda u}{
z}}=\sum_{r=0}^\infty i\sum_{2a+(2b+1)+c=r} \kappaH(a,b,c) z^{-a-c}
\nu^b \lambda^c \tfrac{(2 \pi i u)^{r}}{r!},\]

\section{Relation of Moments to Taylor Coefficients of an Appell-Lerch Sum}
We now introduce a two-variable generating function for the numbers
$N_T(m,n)$:
\[ \widetilde{\mathcal{M}}_T(x,q) := \sum_{n =0}^\infty   \sum_{m \in \Z} N_T(m,n) x^m q^n. \]
In \cite{Ga_gendy}, Garvan has given several descriptions involving
this generating function. In fact, the formula in the next
lemma follows easily from three expressions for
$\widetilde{\mathcal{M}}_T(x,q)$ given by Garvan in formula 4.3 and
4.5 of \cite{Ga_gendy}.
\begin{lem}\label{lem_rel_garvan}
We have
\[ \widetilde{\mathcal{M}}_T(x,q) = \frac{1-x}{(q)_\infty}\sum_{n \in \Z}(-1)^{n}\frac{q^{\frac{n}{2}(Tn+1)}}{1-xq^n} -\frac{1}{(q)_\infty}\sum_{n \in \Z} (-1)^n
q^{\frac{n}{2}(Tn-1)}. \]
\end{lem}
A priori, this result holds as a statement about formal power
series. To see the connection with the moment generating functions
$M_T^r(q)$, we now no longer view $x$ as a formal variable but set
$x:= e^{2 \pi i u}$ for $u \in \C$. Then we find the following
Taylor expansion
\[ \widetilde{\mathcal{M}}_T\left(e^{2 \pi i u},q\right) =\sum_{n =0}^\infty   \sum_{m \in \Z} N_k(m,n)
\sum_{r=0}^\infty \tfrac{(2 \pi i um)^r}{r!} q^n=\sum_{r=0}^\infty
M_T^r(q) \tfrac{(2 \pi i u)^r}{r!}.  \]

Looking at Lemma \ref{lem_rel_garvan}, we see that in the expression
for $\widetilde{\mathcal{M}}_T(x,q)$, the $x$-variable only occurs
in the first summand and not in the second one. Thus except from the
$0$-th moment, the second summand does not affect the
moment-generating functions for higher moments. As we are interested
only in the higher moments\footnote{The behavior for the $0$-th
moment can be treated with classical methods, as Garvan
\cite{Ga_gendy} shows that it is the number of partitions where the
parts satisfy certain congruence conditions.}we will can
equivalently work with the function $\mathcal{M}_T$ defined by
\begin{equation}\label{eqn_rel}
 \mathcal{M}_T(u,q) := \frac{1-e^{2 \pi i u}}{(q)_\infty}\sum_{n \in \Z}(-1)^{n}\frac{q^{\frac{n}{2}(Tn+1)}}{1-e^{2 \pi i u}q^n}
 \end{equation}
for $u,q \in \C$ with $|q|<1$. Then, summarizing the discussion
above, we conclude the following:
\begin{prop}\label{garvan_prop_rel_taylorM}
Let $r>0$. In the Taylor expansion of $\mathcal{M}_T(u,q)$ at $u=0$
the coefficient of $\frac{(2 \pi i u)^r}{r!}$ is equal to the $r$-th
moment generating function
\[ M_T^r(q)=\sum_{n =0}^\infty   m_T^r(n) q^n.\]
\end{prop}
\begin{proof}
In fact our reasoning above is only valid if the representation
\eqnref{eqn_rel} converges. Indeed, it does not converge for all $u
\in \C$. However, it suffices to know that $\mathcal{M}_T(u,q)$ is
defined in a neighborhood of $u=0$. This is immediate from the
following relation between $\mathcal{M}_T$ and Zwegers' function
$\mathrm{A}_T$.
\[ \mathcal{M}_T(u,q) =
\frac{\left(1-e^{2 \pi i u}\right)q^\frac{1}{24}}{\eta(iz)}e^{- \pi
i u T}\mathrm{A}_T\left(u,-\tfrac{T-1}{2}iz;iz\right).\]
\end{proof}
It will turn out to be easier not to work with $\mathcal{M}_T(u,q)$
directly, but to replace the higher level Appell-Lerch sum $A_T$ by
a sum of simpler Appell-Lerch sums using Lemma \ref{pre_fact_pre_AT}. We have
\[ \mathcal{M}_T(u,q) = \sum_{t=-\frac{T-1}{2}}^{\frac{T-1}{2}} \mathcal{C}_{T,t}(u,q), \]
where
\[  \mathcal{C}_{T,t}(u,q):=-\frac{2i \sin(\pi u)q^\frac{1}{24}}{\eta(iz)} e^{2 \pi i u t} \mathrm{A}\left(Tu,tiz; Tiz\right).\]
We express the functions $\mathcal{C}_{T,t}$ in yet another way,
which already indicates that the behavior is rather different
depending on whether $t=0$ or $t\neq0$.
\begin{prop}\label{garvan_prop_rel_C}
We have
\[
\mathcal{C}_{T,t}(u,q)=-\frac{2i \sin(\pi
u)q^\frac{1}{24}}{\eta(iz)} e^{2 \pi i u t} \theta(tiz;
Tiz)\mu(Tu,tiz; Tiz) \quad \text{ and } \quad\mathcal{C}_{T,0}(u,q)
=-\frac{2 \sin(\pi
u)q^\frac{1}{24}\eta^3(Tiz)}{\eta(iz)\theta(Tu;Tiz)}.\]
\end{prop}
\begin{proof}
The first statement is clear by the definition of $\mu$ (see
equation \eqnref{pre_eqn_pre_zwegersmu} in the preliminaries). The
second statement can be deduced from the following identity in
Ramanujans Lost Notebook (see on page 264 of \cite{lostnotebook})
\[ \prod_{n=1} \frac{(1-q^n)^2}{(1-e^{2 \pi i u}q^n)(1-e^{-2 \pi i u}q^n)}=\left(1-e^{2 \pi i u}\right)\sum_{n \in \Z} (-1)^n \frac{q^{n(n+1)/2}}{1-e^{2 \pi i u}q^n},\]
and the product formula of the Jacobi theta function.
\end{proof}

\section{Transformation laws for the functions $\mathcal{C}_{T,t}$}
Our next task is to work out transformation laws for
$\mathcal{C}_{T,t}$. Our treatment here parallels the approach in
\cite{BrMaRh_stat}. In our case there does not occur any new
difficulty, only the exposition gets more involved. For that reason
we omit a full proof and only briefly indicate how one obtains the
transformation laws. A detailed analysis is given in the authors PhD
thesis \cite{Wal_thesis}.

For $t=0$ we deduce the transformation law for $\mathcal{C}_{T,t}$
directly from those for $\eta$ and $\theta$.
\begin{prop}\label{garvan_prop_trans_modular}
Let $T>0$ be an odd integer and suppose that $h,k$, and $z$ satisfy
the usual conditions. Then, we have
\[\mathcal{C}_{T,0}\left(u,e^{\frac{2\pi
i}{k}(h+iz)}\right)=-\frac{2}{\gammaCo }\sqrt{\frac{i}{z}}\frac{
\sin(\pi u)e^{\frac{\pi kT u^2}{z}}e^{\frac{\pi
i}{12k}(h+iz)}}{\chi(h,[-h]_k,k) \eta\left(
\frac{1}{k}([-h]_k+\frac{i}{z})\right)}\frac{ \eta^3\left(
\frac{T}{\gammaCo k}([-\gammaCo
h]_{\frac{k}{\gammaGCD}}+\frac{i}{\gammaCo z})\right)}{
 \theta\left(\frac{iuT}{\gammaCo z};\frac{T}{\gammaCo k}([-\gammaCo h]_{\frac{k}{\gammaGCD}}+\frac{i}{\gammaCo z})\right)}.\]
\end{prop}

For $t \neq 0$ the transformation behavior of
\[
\mathcal{C}_{T,t}(u,q)=-\frac{2i \sin(\pi
u)q^\frac{1}{24}}{\eta(iz)} e^{2 \pi i u t} \theta(tiz;
Tiz)\mu(Tu,tiz; Tiz)\] is no longer that of a Jacobi form but rather
of a mock Jacobi form because of the appearance of the
$\mu$-function. As the transformation law of the theta function is
well-known we only have to find that of the $\mu$-function.

To determine the the transformation law for the $\mu$-function we
argue as in \cite{BrMaRh_stat}. First we write
\begin{equation}\label{garvan_eqn_trans_muhatdecomp}
 \mu\left(u,\tfrac{t}{Tk}(h+iz); \tfrac{1}{k}(h+iz)\right)=\widehat{\mu}\left(u,\tfrac{t}{Tk}(h+iz); \tfrac{1}{k}(h+iz)\right)-\tfrac{i}{2}R\left(u-\tfrac{t}{Tk}(h+iz); \tfrac{1}{k}(h+iz)\right).
\end{equation}
To the functions on the right hand side of
\eqnref{garvan_eqn_trans_muhatdecomp}, we can apply the
transformation properties as stated in the preliminaries. However,
in the transformation laws for both summands there will also occur
non-holomorphic terms involving the $R$-function. On the other hand,
we know that $\mu$ is a holomorphic function. This implies that the
non-holomorphic terms appearing in the transformation laws of both
summands in \eqnref{garvan_eqn_trans_muhatdecomp} have to cancel.
The main technical issue is now to identify these non-holomorphic
terms and prove that they indeed cancel. To achieve this, one has make further use of the transformation
laws for $R$ and $\widehat{\mu}$ in order to obtain non-holomorphic parts of a similar shape. Indeed,
one proves that
\begin{equation*}
\begin{aligned}
R&\left(u-\tfrac{t}{Tk}(h+iz); \tfrac{1}{k}(h+iz)\right)=
\sqrt{\tfrac{i}{kz}} e^{-\frac{\pi t^2 z}{T^2k}+\frac{\pi k}{z}
\left(u- \frac{\rhoT(th)}{Tk}\right)^2-\frac{2 \pi i u
t}{T}}\\
&\times\sum_{l=0}^{k-1}U_H(T,t,l,h,k)\left(R\left(\tfrac{iu}{z}-
\tfrac{\rhoT(th)i}{Tkz}-\alpha_{T,t}(l,k); \tfrac{i}{kz}\right)
-H\left(\tfrac{iu}{z}- \tfrac{\rhoT(th)i}{Tkz}-\alpha_{T,t}(l,k);
\tfrac{i}{kz}\right)\right).
\end{aligned}
\end{equation*}
and
\begin{equation*}
\begin{aligned}
\widehat{\mu}&(u,\tfrac{t}{Tk}(h+iz);
\tfrac{1}{k}(h+iz))=
 \sqrt{\tfrac{i}{z}} e^{\frac{\pi
k}{z}\left(u-\frac{\rhoT(th)}{Tk}\right)^2-\frac{t^2\pi z}{T^2k}-\frac{2 \pi i
u t}{T} }\\
& \times U_\mu(T,t,h,k)
\widehat{\mu}\left(\tfrac{iu}{z},\tfrac{\rhoT(th)}{Tk}\left([-h]_k+\tfrac{i}{z}\right)-\tfrac{t}{Tk}\left(1+h[-h]_k\right);\tfrac{1}{k}\left([-h]_k+\tfrac{i}{z}\right)\right).
\end{aligned}
\end{equation*}
Now, in both equations there occur non-holomorphic parts coming from
the $R$-function (in the second equation these come from the
definition of $\widehat{\mu}$). Now we can employ the same argument
as in \cite{BrMaRh_stat} involving the Fourier-Whittaker expansion
of the $R$-functions in order to prove that these non-holomorphic
terms cancel in \eqnref{garvan_eqn_trans_muhatdecomp}. Working this
out, we obtain a transformation formula for
$\mu\left(u,\tfrac{t}{T}iz; iz\right)$ and use it to find the
following transformation law for $\mathcal{C}_{T,t}$.

\begin{prop}\label{garvan_prop_trans_Ctneq0}
Let $T>0$ be an odd integer and suppose that $h,k$, and $z$ satisfy
the usual conditions. For $t \neq 0$, we have
\[
\mathcal{C}_{T,t}\left(u,e^{\frac{2 \pi
i}{k}(h+iz)}\right)=-\frac{2i \sin(\pi u)e^{\frac{\pi
i}{12k}(h+iz)}}{\eta(\frac{1}{k}(h+iz))}
\theta\left(\tfrac{t}{k}(h+iz);
\tfrac{T}{k}(h+iz)\right)\sqrt{\tfrac{i}{\gammaCo z}} e^{\frac{\pi
k}{Tz}\left(Tu-\frac{\rhoT(t\gammaCo h)}{\gammaCo
k}\right)^2-\frac{t^2\pi z}{Tk} }\]
\[ \left( U_\mu\left(T,t,\gammaCo h,\tfrac{k}{\gammaGCD}\right)
\mu\left(\tfrac{iuT}{\gammaCo z},\tfrac{\rhoT(t\gammaCo h)}{\gammaCo
k}([-\gammaCo h]_{\frac{k}{\gammaGCD}}+\tfrac{i}{\gammaCo
z})-\tfrac{t}{\gammaCo k} (1+\gammaCo h[-\gammaCo
h]_{\frac{k}{\gammaGCD}});\tfrac{\gammaGCD}{k}([-\gammaCo
h]_{\frac{k}{\gammaGCD}}+\tfrac{i}{\gammaCo
z})\right)\phantom{\sum_{l=0}^{\frac{k}{\gammaGCD}-1}}\right.
\]
\[ \left.\phantom{aaaaaa}+\tfrac{i}{2\sqrt{\frac{k \gammaCo}{T}}} \sum_{l=0}^{\frac{k \gammaCo}{T}-1}U_H\left(T,t,l,\gammaCo h,\tfrac{k \gammaCo}{T}\right)H\left(\tfrac{iuT}{\gammaCo z}-
\tfrac{\rhoT(t\gammaCo h)i}{\gammaCo^2 k z}-\alpha_{T,t}(l,\tfrac{k \gammaCo}{T});
\tfrac{Ti}{\gammaCo^2 k z}\right)\right).\]
\end{prop}

\section{Asymptotic Expansions for the Moment Generating Functions $M^r_T(q)$}
By Proposition
\ref{garvan_prop_rel_taylorM} we know that, for $r>0$, the $r$-th
Taylor coefficient in the Taylor expansion of
\begin{equation}\label{garvan_eqn_moment_mathcalM} \mathcal{M}_T\left(u,e^{\frac{2\pi i}{k}(h+iz)}\right)=\sum_{t=-\frac{T-1}{2}}^{\frac{T-1}{2}} \mathcal{C}_{T,t}\left(u,e^{\frac{2\pi i}{k}(h+iz)}\right)
\end{equation}
at $u=0$ is equal to $M_T^r\left(e^{\frac{2\pi i}{k}(h+iz)}\right)$.
If we now apply the transformation properties of Proposition
\ref{garvan_prop_trans_modular} and Proposition
\ref{garvan_prop_trans_Ctneq0} to the functions appearing on the
right hand side of \eqnref{garvan_eqn_moment_mathcalM} we obtain
\begin{equation}\label{garvan_eqn_moment_mathcalM2} \mathcal{M}_T\left(u,e^{\frac{2\pi i}{k}(h+iz)}\right) =
\mathcal{C}_{T,0}\left(u,e^{\frac{2\pi i}{k}(h+iz)}\right)+
\sum_{\substack{t=-\frac{T-1}{2}\\t \neq 0}}^{\frac{T-1}{2}}
\left(\mathcal{C}^\mu_{T,t}\left(u,e^{\frac{2\pi
i}{k}(h+iz)}\right)+ \mathcal{C}^H_{T,t}\left(u,e^{\frac{2\pi
i}{k}(h+iz)}\right)\right),
\end{equation}
where $\mathcal{C}_{T,0}$, $\mathcal{C}^\mu_{T,t}$ and $\mathcal{C}^H_{T,t}$ are defined by
\begin{equation}\label{garvan_eqn_moment_mod}\mathcal{C}_{T,0}\left(u,e^{\frac{2\pi
i}{k}(h+iz)}\right):=-\frac{2}{\gammaCo}\sqrt{\frac{i}{z}}
e^{\frac{\pi k u^2T}{z}}\frac{\sin(\pi u)e^{\frac{\pi
i}{12k}(h+iz)}}{\chi\left(h,[-h]_k,k\right) \eta\left(
\frac{1}{k}\left([-h]_k+\frac{i}{z}\right)\right)}\frac{
\eta^3\left( \frac{T}{\gammaCo k}\left([-\gammaCo
h]_{\frac{k}{\gammaGCD}}+\frac{i}{\gammaCo z}\right)\right)}{
 \theta\left(\frac{iuT}{\gammaCo z};\frac{T}{\gammaCo k}\left([-\gammaCo h]_{\frac{k}{\gammaGCD}}+\frac{i}{\gammaCo z}\right)\right)},
\end{equation}
\begin{equation}\label{garvan_eqn_moment_mu}
\begin{aligned} \mathcal{C}^\mu_{T,t}&\left(u,e^{\frac{2\pi
i}{k}(h+iz)}\right):=-\sqrt{\frac{i}{\gammaCo z}}\frac{2i \sin(\pi
u)e^{\frac{\pi i}{12k}(h+iz) }}{\eta\left(\frac{1}{k}(h+iz)\right)}
 \theta\left(\tfrac{t}{k}(h+iz); \tfrac{T}{k}(h+iz)\right)
 e^{\frac{\pi
k}{Tz}\left(Tu-\frac{\rhoT(t\gammaCo h)}{\gammaCo
k}\right)^2-\frac{t^2\pi
z}{Tk}}\\
& \phantom{aaaa}\times U_\mu\left(T,t,\gammaCo h,\tfrac{k}{\gammaGCD}\right)
\mu\left(\tfrac{iuT}{\gammaCo z},\tfrac{\rhoT(t\gammaCo h)}{\gammaCo
k}\left([-\gammaCo h]_{\frac{k}{\gammaGCD}}+\tfrac{i}{\gammaCo
z}\right)-\tfrac{t}{\gammaCo k} \left(1+\gammaCo h[-\gammaCo
h]_{\frac{k}{\gammaGCD}}\right);\tfrac{\gammaGCD}{k}\left([-\gammaCo
h]_{\frac{k}{\gammaGCD}}+\tfrac{i}{\gammaCo z}\right)\right),
\end{aligned}
\end{equation}
and
\begin{equation}\label{garvan_eqn_moment_mord}
\begin{aligned}
\mathcal{C}^H_{T,t}\left(u,e^{\frac{2\pi
i}{k}(h+iz)}\right)&:=\frac{1}{\gammaCo \sqrt{\frac{k}{T}}}
\sqrt{\frac{i}{z}}\frac{ \sin(\pi u)e^{\frac{\pi
i}{12k}(h+iz)}}{\eta\left(\frac{1}{k}(h+iz)\right)}
\theta\left(\tfrac{t}{k}(h+iz); \tfrac{T}{k}(h+iz)\right)
e^{\frac{\pi k}{Tz}\left(Tu-\frac{\rhoT(t\gammaCo h)}{\gammaCo
k}\right)^2-\frac{t^2\pi
z}{Tk} }\\
& \times \sum_{l=0}^{\frac{k \gammaCo}{T}-1}U_H\left(T,t,l,\gammaCo
h,\tfrac{k \gammaCo}{T}\right)H\left(\tfrac{iuT}{\gammaCo z}-
\tfrac{\rhoT(t\gammaCo h)i}{\gammaCo^2 k z}-\alpha_{T,t}\left(l,\tfrac{k
\gammaCo}{T}\right); \tfrac{Ti}{\gammaCo^2 k z}\right).
\end{aligned}
\end{equation}

Now, in order to find an expression for $M_T^r\left(e^{\frac{2\pi
i}{k}(h+iz)}\right)$, we will have to determine Taylor expansions of
these three expressions with respect to $u$. As this will give very
messy formulas and since we are only interested in asymptotic
expressions anyways we content ourself with asymptotic Taylor
expansions. By this we mean that in the Taylor expansion with
respect to $u$ we split the Taylor coefficients (which will be
functions in $z$) into a main part and an error part.

As in \cite{BrMaRh_stat}, the contributions coming from
$\mathcal{C}_{T,0}$ and  $\mathcal{C}^\mu_{T,t}$ can be jointly
expressed in a nice way. In \cite{BrMaRh_stat} there was no need to
consider the contribution from $\mathcal{C}^H_{T,t}$ explicitly, as
it was part of the error term. In the case $T>3$, we can no longer
neglect these contributions and have to determine an asymptotic
Taylor expansion, as well.
\subsection{The contribution from $\mathcal{C}_{T,0}$ and  $\mathcal{C}^\mu_{T,t}$}
In order to derive asymptotic Taylor expansions for
$\mathcal{C}_{T,0}$ and  $\mathcal{C}^\mu_{T,t}$, we have to first
find asymptotic Taylor expansions for the $\theta$-function and the
$\mu$-function appearing in \eqnref{garvan_eqn_moment_mod} and
\eqnref{garvan_eqn_moment_mu}, respectively. For that purpose we
first use that the $\mu$-function is the quotient of $\mathrm{A}$
and a $\theta$-function. To find asymptotic Taylor expansions
$\theta$ and $\mathrm{A}$ we consider their series representations,
identify the main parts and bound the other summands into the error
term. Working this out is rather straightforward but involves a few
lengthy computations. We omit the proof and refer Section 2 of Part
B of the author's PhD thesis \cite{Wal_thesis} for a more detailed
analysis.

Suppose that $T>0$ is and odd integer and suppose $h,k$, and $z$
satisfy the usual conditions. Furthermore, assume that
$\Re{\frac{1}{z}} \geq \frac{k}{2}$. Then,
\[ \theta\left(\tfrac{iuT}{\gammaCo z};\tfrac{T}{\gammaCo k}\left([-\gammaCo h]_{\frac{k}{\gammaGCD}}+\tfrac{i}{\gammaCo z}\right)\right)^{-1}  = i \frac{e^{-\frac{\pi iT}{4\gammaCo k}\left([-\gammaCo
h]_{\frac{k}{\gammaGCD}}+\frac{i}{\gammaCo z}\right)}}{2 \sinh
\left(\frac{ T \pi u}{\gammaCo z}\right)}+
 \frac{1}{u}\sum_{r=0}^\infty c_{r,T,h,k}(z)\tfrac{(2\pi i u)^r}{r!},\] where $c_{r,T,h,k}:\HR \to \C$ are functions which satisfy $|c_{r,T,h,k}(z)| \ll_{r,T} |z|^{1-r} e^{-\frac{7T\pi}{4\gammaCo^2k}  \Re{\frac{1}{z}}}.$

Similarly, we may decompose the Appell-Lerch sum
\begin{equation}\label{garvan_eqn_genasym_applerch}
\mathrm{A}\left(\tfrac{iuT}{\gammaCo z},\tfrac{\rhoT(t\gammaCo
h)}{\gammaCo k}\left([-\gammaCo
h]_{\frac{k}{\gammaGCD}}+\tfrac{i}{\gammaCo
z}\right)-\tfrac{t}{\gammaCo k} \left(1+\gammaCo h[-\gammaCo
h]_{\frac{k}{\gammaGCD}}\right);\tfrac{\gammaGCD}{k}\left([-\gammaCo
h]_{\frac{k}{\gammaGCD}}+\tfrac{i}{\gammaCo z}\right)\right).
\end{equation}
as
\[ \frac{1}{2\sinh\left(\frac{\pi u T}{\gammaCo z}\right)}+
\sum_{r=0}^\infty c_{r,T,t,h,k}(z) \tfrac{(2 \pi iu)^r}{r!}\] with
certain functions $c_{r,t,h,k,T}:\HR \to \C$ that satisfy $|c_{r,t,h,k,T}(z)| \ll_{r,T} |z|^{-r} e^{-\frac{2 \pi}{\gammaCo^2 k }\left(T-|\rhoT(t\gammaCo h)|\right)\Re{ \frac{1}{z}}}.$

In order to write down an asymptotic Taylor expansion for
$\mathcal{C}_{T,0}$ and  $\mathcal{C}^\mu_{T,t}$ it remains to
combine the asymptotic Taylor expansion above with the asymptotic
behavior of the $\eta$- and $\theta$-functions appearing in
\eqnref{garvan_eqn_moment_mod} and \eqnref{garvan_eqn_moment_mu}.
Doing so, we obtain the following two propositions.
\begin{prop}\label{garvan_prop_form_CmodAsym}
Let $T>0$ be an odd integer and suppose that $h,k$, and $z$ satisfy
the usual conditions. Furthermore, assume that $\Re{\frac{1}{z}}
\geq \frac{k}{2}$. Then,
\[\mathcal{C}_{T,0}\left(u,e^{\frac{2\pi
i}{k}(h+iz)}\right)=-\frac{i}{\gammaCo}\sqrt{\frac{i}{z}}e^{\frac{\pi
k u^2T}{z}}\frac{ \sin(\pi u)e^{\frac{\pi
i}{12k}(h+iz)}e^{-\frac{\pi i}{ 12 k
}\left([-h]_k+\frac{i}{z}\right)}}{\sinh \left(\frac{T \pi
u}{\gammaCo z}\right)\chi\left(h,[-h]_k,k\right)} +
\sum_{r=0}^\infty c_{r,T,0,h,k}(z)\tfrac{(2\pi i u)^r}{r!},\] with
functions $c_{r,T,h,k}:\HR \to \C$ which satisfy
\[ |c_{2,T,0,h,k}(z)| \ll_{T} |z|^{-\frac{1}{2}} e^{-\frac{\pi}{k}\left(\frac{5T}{2\gammaCo^2}-\frac{1}{12}\right) \Re{\frac{1}{z}}} \qquad
\text{and} \qquad |c_{r,T,0,h,k}(z)| \ll_{r,T} k^\frac{r}{2}
|z|^{\frac{1}{2}-r}
e^{-\frac{\pi}{k}\left(\frac{5T}{2\gammaCo^2}-\frac{1}{12}\right)
\Re{\frac{1}{z}}}. \]
\end{prop}
\begin{prop}\label{garvan_prop_form_CmuAsym}
Let $T>0$ be an odd integer and suppose that $h,k$, and $z$ satisfy the usual conditions. Furthermore, assume that $\Re{\frac{1}{z}} \geq \frac{k}{2}$. For any $t \neq 0$, we find that
\[ \mathcal{C}^\mu_{T,t}\left(u,e^{\frac{2\pi
i}{k}(h+iz)}\right)= -\frac{i}{\gammaCo }\sqrt{\frac{i}{z}}
\frac{\sin(\pi u)e^{\frac{\pi i}{12k}(h+iz)e^{-\frac{\pi i}{ 12 k
}\left([-h]_k+\frac{i}{z}\right)}}}{\sinh \left(\frac{ T \pi
u}{\gammaCo z}\right)\chi\left(h,[-h]_k,k\right)} e^{\frac{\pi k T
u^2}{z}-\frac{2\pi u\rhoT(t\gammaCo h) }{\gammaCo z}}+
\sum_{r=0}^\infty c_{r,T,t,h,k}(z) \tfrac{(2\pi i u)^r}{r!}, \]
where $c_{r,T,t,h,k}: \HR \to \C$ are functions that satisfy
\[  |c_{2,T,t,h,k}(z)| \ll_{T} |z|^{-\frac{3}{2}} e^{-\frac{2 \pi}{ k } \Re{\frac{1}{z}}\left(\frac{T-|\rhoT(t\gammaCo h)|}{\gammaCo^2}-\frac{1}{24}\right)}\]
and, for $r>2$,
\[  |c_{r,T,t,h,k}(z)| \ll_{r,T} k^\frac{r}{2}|z|^{\frac{1}{2}-r} e^{-\frac{2 \pi}{ k } \Re{\frac{1}{z}}\left(\frac{T-|\rhoT(t\gammaCo h)|}{\gammaCo^2}-\frac{1}{24}\right)}.\]
\end{prop}
We see that the main terms of the asymptotic expressions in
Proposition \ref{garvan_prop_form_CmodAsym} and Proposition
\ref{garvan_prop_form_CmuAsym} are very similar, and that it might
be possible to combine all contributions in
\eqnref{garvan_eqn_moment_mathcalM2} coming from $\mathcal{C}_{T,0}$
and $\mathcal{C}_{T,t}^\mu$ into one single formula. In fact, this
reduces to evaluate the sum
\[ \frac{1}{\gammaCo \sinh \left(\frac{\pi u T}{\gammaCo  z}\right)}\sum_{t=-\frac{T-1}{2}}^{\frac{T-1}{2}} e^{-\frac{2\pi
u\rhoT(t\gammaCo h) }{\gammaCo z}}.\] Using the definition of the
$\rhoT$-function, it is easy to deduce that this expression is
always (i.e. independently of $k$) equal to to
$\frac{1}{\sinh\left(\frac{\pi u }{z}\right)}$. Furthermore, the
error terms in Proposition \ref{garvan_prop_form_CmodAsym} and
Proposition \ref{garvan_prop_form_CmuAsym} are also of a similar
shape. Indeed, using the fact that $\frac{5 T}{4 \gammaCo^2} \geq
\frac{1}{T}$ and $\frac{T-|\rhoT(t\gammaCo h)|}{\gammaCo^2} \geq
\frac{1}{T}$, we easily arrive at the following result.
\begin{prop}\label{garvan_prop_moment_CmumodMainAsym}
Let $T>0$ be an odd integer and suppose that $h,k,$ and $z$ satisfy
the usual conditions. Furthermore, assume that $\Re{\frac{1}{z}}
\geq \frac{k}{2}$. Then, we have
\begin{align*} \mathcal{C}_{T,0}&\left(u,e^{\frac{2\pi
i}{k}(h+iz)}\right)+ \sum_{\substack{t=-\frac{T-1}{2} \\ t \neq
0}}^{\frac{T-1}{2}}\mathcal{C}^\mu_{T,t}\left(u,e^{\frac{2\pi
i}{k}(h+iz)}\right)\\
&=-i^{\frac{3}{2}}\frac{1}{\sqrt{z}}\frac{\sin(\pi
u)}{\sinh\left(\frac{\pi u}{z}\right)}e^{\frac{\pi i}{12k}(h-[h]_k)
}\chi^{-1}\left(h,[-h]_k,k\right)e^{-\frac{\pi
}{12k}\left(1-\frac{1}{z}\right) } e^{\frac{\pi k T u^2}{z}} +
\sum_{r=0}^\infty c_{r,T,h,k}(z) \tfrac{(2 \pi i u)^r}{r!},
\end{align*}
where $c_{r,T,h,k}:\HR \to \C$ are functions that satisfy
\[  |c_{2,T,h,k}(z)| \ll_{T} |z|^{-\frac{3}{2}}  e^{-\frac{2 \pi}{ k }
\Re{\frac{1}{z}}\left(\frac{1}{T}-\frac{1}{24}\right)} \qquad
\text{and} \qquad  |c_{r,T,h,k}(z)| \ll_{r,T}
k^\frac{r}{2}|z|^{\frac{1}{2}-r} e^{-\frac{2 \pi}{ k }
\Re{\frac{1}{z}}\left(\frac{1}{T}-\frac{1}{24}\right)}.\] The main
term may be rewritten as
\[-i^{\frac{3}{2}} e^{\frac{\pi i}{12k}(h-[h]_k)
}\chi^{-1}\left(h,[-h]_k,k\right)e^{-\frac{\pi
}{12k}\left(z-\frac{1}{z}\right) } \sum_{r=0}^\infty
\sum_{2a+2b+2c=r} \kappa(a,b,c) (kT)^a z^{\frac{1}{2}-a-2c}
\tfrac{(2 \pi i u)^r}{r!}. \]
\end{prop}
\subsection{The contribution from $\mathcal{C}^H_{T,t}$}
In order to obtain an asymptotic Taylor expansion of
\eqnref{garvan_eqn_moment_mord} we will rewrite this expression
first. Grouping all terms that depend on $u$, we can write
\eqnref{garvan_eqn_moment_mord} as follows:
\begin{equation}\label{garvan_eqn_moment_mordell1}
\begin{aligned}
&\mathcal{C}^H_{T,t}\left(u,e^{\frac{2\pi
i}{k}(h+iz)}\right)=\tfrac{1}{\gammaCo \sqrt{\frac{k}{T}}}
\sqrt{\frac{i}{z}} \frac{ e^{\frac{\pi
i}{12k}(h+iz)}}{\eta\left(\frac{1}{k}(h+iz)\right)}
\theta\left(\tfrac{t}{k}(h+iz); \tfrac{T}{k}(h+iz)\right)
e^{\frac{\pi \rhoT(t\gammaCo h)^2}{\gammaCo^2kTz}-\frac{t^2\pi
z}{Tk} } \\
&\times \sum_{l=0}^{\frac{k\gammaCo }{T}-1}U_H\left(T,t,l,\gammaCo
h,\tfrac{k \gammaCo}{T}\right)\sin(\pi u)e^{\frac{\pi k T u^2}{z}}
e^{-\frac{2 \pi u \rhoT(t\gammaCo h)}{\gammaCo
z}}H\left(\tfrac{iuT}{\gammaCo z}- \tfrac{\rhoT(t\gammaCo
h)i}{\gammaCo^2 k z}-\alpha_{T,t}\left(l,\tfrac{k
\gammaCo}{T}\right); \tfrac{Ti}{\gammaCo^2 k z}\right).
\end{aligned}
\end{equation}

We can now find the Taylor expansion with respect to $u$ by simply
moving all the functions depending on $u$ into the defining integral
of the Mordell-integral $H$, computing the Taylor expansion inside
the integral and interchanging summation and integration (which needs
to and can be justified). In order to state our result, we define
\begin{equation}\label{garvan_eqn_genasym_HTCdef} \mathcal{H}_{c,T}(\alpha,\gamma,\varrho,k;z):=\int_{\R}
(x+i\varrho)^c \frac{e^{-\frac{\pi T x^2}{\gamma^2 k z}+2 \pi x
\alpha}}{\cosh(\pi (x+i \varrho))} dx.
\end{equation}
\begin{lem}\label{garvan_prop_genasym_HTaylor}
Let $T>0$ be an odd integer. Further suppose that
$\alpha,\gamma,\varrho \in \Q$ and that $k$ and $z$ satisfy the
usual conditions. Then, we have the following Taylor expansion at
$u=0$:
\[ \sin(\pi u) e^{\frac{\pi k T u^2}{z}} e^{-\frac{2 \pi u \varrho T}{\gamma
z}}H\left(\tfrac{Tiu}{\gamma z}- \tfrac{ T i \varrho}{\gamma^2
kz}-\alpha; \tfrac{Ti}{\gamma^2 kz}\right)
\]
\[ =  e^{-\frac{\pi T
\varrho^2}{\gamma^2 k z}+2\pi i \varrho \alpha }\sum_{r=0}^\infty
\sum_{2a+(2b+1)+c=r}\left( i\kappaH(a,b,c) z^{-a-c} (kT)^b
\left(\tfrac{T}{\gamma}\right)^c
\mathcal{H}_{c,T}(\alpha,\gamma,\varrho,k;z) \right) \tfrac{(2 \pi i
u)^{r}}{r!}.\]
\end{lem}

Combining Lemma \ref{garvan_prop_genasym_HTaylor} (with parameters
$\gamma=\gammaCo, \alpha=\alpha_{T,t}\left(l,\tfrac{k\gammaCo
}{T}\right), \varrho=\tfrac{\rhoT(t\gammaCo h)}{T}$) with
\eqnref{garvan_eqn_moment_mordell1}, we obtain that
$\mathcal{C}^H_{T,t}\left(u,e^{\frac{2\pi i}{k}(h+iz)}\right)$ is
equal to
\begin{equation}\label{garvan_eqn_moment_mordell2}
\begin{aligned}
&\tfrac{1}{\gammaCo \sqrt{\frac{k}{T}}} \sqrt{\frac{i}{z}} \frac{
e^{\frac{\pi i}{12k}(h+iz)}}{\eta\left(\frac{1}{k}(h+iz)\right)}
\theta\left(\tfrac{t}{k}(h+iz); \tfrac{T}{k}(h+iz)\right)
e^{-\frac{t^2\pi z}{Tk} } \sum_{l=0}^{\frac{k\gammaCo
}{T}-1}U_H\left(T,t,l,\gammaCo h,\tfrac{k \gammaCo}{T}\right)e^{2\pi
i \frac{\rhoT(t\gammaCo h)}{T}
\alpha_{T,t}\left(l,\frac{k\gammaCo }{T}\right) } \\
&\times \sum_{r=0}^\infty \sum_{2a+(2b+1)+c=r}\left( i\kappaH(a,b,c)
z^{-a-c} (kT)^b \left(\tfrac{T}{\gamma}\right)^c
\mathcal{H}_{c,T}\left(\alpha_{T,t}\left(l,\tfrac{k\gammaCo
}{T}\right),\gammaCo ,\tfrac{\rhoT(t\gammaCo h)}{T},k;z\right)
\right) \tfrac{(2 \pi i u)^{r}}{r!}.
\end{aligned}
\end{equation}


Now we can determine the asymptotic Taylor expansion of
\eqnref{garvan_eqn_moment_mord}.
\begin{prop}\label{garvan_prop_moment_mordexpWithoutError}
Let $T>0$ be an odd integer and suppose that $h,k$, and $z$ satisfy
the usual conditions. Suppose that $\Re{\frac{1}{z}} \geq
\frac{k}{2}$. Let $c$ and $r$ be positive integers with $c \leq r$
and $t \neq 0$. Then,
\begin{align*}
\mathcal{C}^H_{T,t}&\left(u,e^{\frac{2\pi
i}{k}(h+iz)}\right)=\gammaCo^{-\frac{3}{2}} \sqrt{\tfrac{T}{k}}
e^{-\frac{\pi z}{12k}}
e^{-\frac{\pi}{\gammaCo^2Tkz}\left(|\rhoT(t\gammaCo
h)|^2+\frac{T^2}{4}-|\rhoT(t\gammaCo h)|T\right)+\frac{\pi }{12 kz
}} \sum_{l=0}^{\frac{k\gammaCo
}{T}-1}U_H^{\ast}(T,t,l,h,k)\\
& \times \sum_{r=0}^\infty \sum_{2a+(2b+1)+c=r} \kappaH(a,b,c)
z^{-\frac{1}{2}-a-c} (kT)^b \left(\tfrac{T}{\gamma}\right)^c
\mathcal{H}_{c,T}\left(\alpha_{T,t}\left(l,\tfrac{k\gammaCo
}{T}\right),\gammaCo
,\tfrac{\rhoT(t\gammaCo h)}{T},k;z\right)  \tfrac{(2 \pi i u)^{r}}{r!}\\
&+\sum_{r=0}^\infty c_{r,T,t,h,k}(z) \tfrac{(2 \pi i u)^r}{r!},
\end{align*}
where $c_{r,T,t,h,k}:\HR \to \C$ are functions that satisfy
\[ |c_{2,T,t,h,k}(z)| \ll_{r,T} k^{\frac{1}{2}}  |z|^{-\frac{3}{2}} e^{- \frac{\pi}{k}\left(\frac{9}{4T}
-\frac{1}{12}\right)\Re{\frac{1}{z}}}\     \ \text{and} \ \
|c_{r,T,t,h,k}(z)| \ll_{r,T} k^{\frac{r}{2}}  |z|^{-r + \frac{1}{2}}
e^{- \frac{\pi}{k}\left(\frac{9}{4T}
-\frac{1}{12}\right)\Re{\frac{1}{z}}}.\]
\end{prop}
\begin{proof}
First, we investigate the asymptotic behavior of the quotient of the
$\theta$-function and $\eta$-function in
\eqnref{garvan_eqn_moment_mordell2}. In order to do that, we apply the transformation formulas for $\eta$ and $\theta$ and then consider the defining series of the $\eta$-function and $\theta$-function. We extract the
leading term and bound the other terms into an error part. This yields that
\[ \frac{\theta\left(\frac{t}{k}(h+iz);\frac{T}{k}(h+iz)\right)}{\eta\left(\frac{1}{k}(h+iz)\right)}\]
is equal to
\[\frac{1}{\sqrt{\gammaCo }} \frac{U^{\ast}_\theta(T,t,h,k)
e^{\frac{\pi i}{12 k}[-h]_k}}{\chi(h,[-h]_k,k)}
e^{-\frac{\pi}{\gammaCo ^2Tkz}\left(\rhoT(t\gammaCo
h)^2+\frac{T^2}{4}-|\rhoT(t\gammaCo h)|T \right)+\frac{1}{12kz}}
e^\frac{\pi t^2z}{Tk}+ O_T\left(e^{- \frac{9\pi}{4Tk}
\Re{\frac{1}{z}}+\frac{\pi}{12k}\Re{\frac{1}{z}}}\right).\] Now the
main part is easily read off and it remains to bound the error term.
Taking absolute values, we see that the error term
$|c_{r,T,t,h,k}(z)|$ is essentially bounded by
\begin{align*} & \left|\tfrac{1}{\gammaCo
\sqrt{\frac{k}{T}}} \sqrt{\tfrac{i}{z}} e^{\frac{\pi i}{12k}(h+iz)}
e^{- \frac{9\pi}{4Tk}
\Re{\frac{1}{z}}+\frac{\pi}{12k}\Re{\frac{1}{z}}} e^{\frac{t^2\pi
z}{Tk} } \sum_{l=0}^{\frac{k\gammaCo }{T}-1}U_H\left(T,t,l,\gammaCo
h,\frac{k \gammaCo}{T}\right)e^{2\pi i \frac{\rhoT(t\gammaCo h)}{T}
\alpha_{T,t}\left(l,\tfrac{k\gammaCo }{T}\right) } \right. \\
& \phantom{aaaaaaaaaaa}\times \left.\sum_{2a+(2b+1)+c=r} \left(
i\kappaH(a,b,c) z^{-a-c} (kT)^b \left(\tfrac{T}{\gamma}\right)^c
\mathcal{H}_{c,T}\left(\alpha_{T,t}\left(l,\tfrac{k\gammaCo
}{T}\right),\gammaCo ,\tfrac{\rhoT(t\gammaCo h)}{T},k;z\right) \right)
\right|.
\end{align*}
Introducing a uniform bound for $\kappa^\ast$ in $r$ and estimating all constants, we
see that
\begin{align*} |c_{r,T,t,h,k}(z)| \ll_{r,T} &e^{- \frac{\pi}{k}\left(\frac{9}{4T}
-\frac{1}{12}\right)\Re{\frac{1}{z}}} \hspace*{-1.5em} \sum_{2a+(2b+1)+c=r}
k^{b-\frac{1}{2}} |z|^{-\frac{1}{2}-a-c}
\sum_{l=0}^{\frac{k\gammaCo }{T}-1}
\left|\mathcal{H}_{c,T}\left(\alpha_{T,t}\left(l,\tfrac{k\gammaCo
}{T}\right),\gammaCo ,\tfrac{\rhoT(t\gammaCo h)}{T},k;z\right)\right|.
\end{align*}
Now, we observe that
\[ \sum_{l=0}^{\frac{k\gammaCo }{T}-1} \left|\mathcal{H}_{c,T}\left(\alpha_{T,t}\left(l,\tfrac{k\gammaCo }{T}\right),\gammaCo
,\tfrac{\rhoT(t\gammaCo h)}{T},k;z\right)\right| \leq  k \max_{c
\leq r, l \leq k}
\left|\mathcal{H}_{c,T}\left(\alpha_{T,t}\left(l,\tfrac{k\gammaCo
}{T}\right),\gammaCo ,\tfrac{\rhoT(t\gammaCo
h)}{T},k;z\right)\right|.\] Using the fact that
$\left|\alpha_{T,t}\left(l,\tfrac{k\gammaCo }{T}\right)\right| <
\frac{1}{2}$ and $\left|\tfrac{\rhoT(t\gammaCo h)}{T}\right|< \frac{1}{2}$, one shows that
\[ \left|\mathcal{H}_{c,T}\left(\alpha_{T,t}\left(l,\tfrac{k\gammaCo }{T}\right),\gammaCo ,\tfrac{\rhoT(t\gammaCo h)}{T},k;z\right)\right| \ll_{c,T} 1 \ll_{r,T} 1.\]
This yields
\[ |c_{r,T,t,h,k}(z)| \ll_{r,T} e^{- \frac{\pi}{k}\left(\frac{9}{4T}
-\frac{1}{12}\right)\Re{\frac{1}{z}}} \sum_{2a+(2b+1)+c=r}
k^{b+\frac{1}{2}} |z|^{-\frac{1}{2}-a-c} .\] We distinguish two
cases. If $r=2$, then we have that $a=b=0$ and $c=1$. Then we obtain
\[ |c_{2,t,T,h,k}(z)| \ll_{r,T} k^{\frac{1}{2}}  |z|^{-\frac{3}{2}} e^{- \frac{\pi}{k}\left(\frac{9}{4T}
-\frac{1}{12}\right)\Re{\frac{1}{z}}}.\] If $r>2$, we bound $k^b$ by
$k^\frac{r-1}{2}$. We also see that $|z|^{-a-c} \leq |z|^{-r+1}$.
\end{proof}
\subsection{Asymptotic expansion of $M_T^r$ at roots of unity}
In Proposition \ref{garvan_prop_moment_CmumodMainAsym} and in Proposition
\ref{garvan_prop_moment_mordexpWithoutError}  we have found
asymptotic Taylor expansions for
\[ \mathcal{C}_{T,0}\left(u,e^{\frac{2\pi i}{k}(h+iz)}\right)+
\sum_{\substack{t=-\frac{T-1}{2}\\t \neq 0}}^{\frac{T-1}{2}}
\mathcal{C}^\mu_{T,t}\left(u,e^{\frac{2\pi i}{k}(h+iz)}\right)
\qquad \text{ and } \qquad \mathcal{C}^H_{T,t}\left(u,e^{\frac{2\pi
i}{k}(h+iz)}\right),\] respectively. Combining these results leads
to an asymptotic Taylor expansion for
$\mathcal{M}_T\left(u,e^{\frac{2\pi i}{k}(h+iz)}\right)$. Referring
to Proposition \ref{garvan_prop_rel_taylorM}, we can then
immediately derive asymptotic formulas for $M_T^r\left(e^{\frac{2\pi
i}{k}(h+iz)}\right)$. To state our result, we first define the
functions
\begin{equation}\label{garvan_eqn_moment_MTRmu} M_T^{r,\mu}(h,k;z):=-i^{\frac{3}{2}} e^{\frac{\pi i}{12k}\left(h-[h]_k\right)
}\chi^{-1}\left(h,[-h]_k,k\right)e^{-\frac{\pi }{12k}\left(z-\frac{1}{z}\right) }
\sum_{2a+2b+2c=r} \kappa(a,b,c) (kT)^a z^{\frac{1}{2}-a-2c}
\end{equation}
and
\begin{equation}\label{garvan_eqn_moment_MTRmordell}
\begin{aligned}
M_T^{r,H}(t,l,h,k;z):=&\gammaCo^{-\frac{3}{2}} \sqrt{\tfrac{T}{k}}
e^{-\frac{\pi z}{12k}}
e^{-\frac{\pi}{\gammaCo^2Tkz}\left(|\rhoT(t\gammaCo
h)|^2+\frac{T^2}{4}-|\rhoT(t\gammaCo h)|T\right)+\frac{\pi }{12 kz
}} \sum_{l=0}^{\frac{k\gammaCo
}{T}-1}U_H^{\ast}(T,t,l,h,k)\\
& \times\sum_{2a+(2b+1)+c=r} \kappaH(a,b,c) z^{-\frac{1}{2}-a-c} (kT)^b
\left(\tfrac{T}{\gamma}\right)^c
\mathcal{H}_{c,T}\left(\alpha_{T,t}\left(l,\tfrac{k\gammaCo }{T}\right),\gammaCo
,\tfrac{\rhoT(t\gammaCo h)}{T},k;z\right).
\end{aligned}
\end{equation}
Now, Proposition \ref{garvan_prop_moment_CmumodMainAsym} and Proposition \ref{garvan_prop_moment_mordexpWithoutError} and \eqnref{garvan_eqn_moment_mathcalM2} imply that
\begin{align*} \mathcal{M}_T\left(u,e^{\frac{2\pi i}{k}(h+iz)}\right)=&  \sum_{r=0}^\infty M_T^{r,\mu}(h,k;z)\tfrac{(2\pi i r)^r}{r!}+
 \sum_{r=0}^\infty E_T^{r,\mu}(h,k;z)\tfrac{(2\pi i r)^r}{r!}\\
&+\sum_{r=0}^\infty \sum_{\substack{t=-\frac{T-1}{2}\\t \neq 0}}^{\frac{T-1}{2}} \sum_{l=0}^{\frac{k \gammaCo}{T}-1} M_T^{r,H}(t,l,h,k;z)\tfrac{(2\pi i r)^r}{r!}+\sum_{r=0}^\infty \sum_{\substack{t=-\frac{T-1}{2}\\t \neq 0}}^{\frac{T-1}{2}} E_T^{r,H}(t,l,h,k;z)\tfrac{(2\pi i r)^r}{r!},
\end{align*}
with certain functions $E_T^{r,\mu}, E_T^{r,H}$. In the variable $z$, the functions $M_T^{r,\mu}, M_T^{r,H},E_T^{r,\mu}$, and $E_T^{r,H}$ are functions on $\HR$ with values in $\C$. The error terms $E_T^{r,\mu}$ and $E_T^{r,H}$ satisfy the bounds
\begin{equation}\label{garvan_eqn_moment_finalErrMu}
\begin{aligned} |E_T^{2,\mu}(h,k;z)| &\ll_{T} |z|^{-\frac{3}{2}} e^{\frac{\pi }{12k}\Re{\frac{1}{z}}\left( \frac{1}{T}-\frac{1}{24} \right)},\\
|E_T^{r,\mu}(h,k;z)| &\ll_{r,T} k^\frac{r}{2}|z|^{-r+\frac{1}{2}} e^{\frac{\pi }{12k}\Re{\frac{1}{z}}\left( \frac{1}{T}-\frac{1}{24} \right)},
\end{aligned}
\end{equation}
and
\begin{equation}\label{garvan_eqn_moment_finalErrH}
\begin{aligned} |E_T^{2,H}(t,h,k;z)| &\ll_{r,T} k^{\frac{1}{2}}  |z|^{-\frac{3}{2}} e^{- \frac{\pi}{k}\left(\frac{9}{4T}
-\frac{1}{12}\right)\Re{\frac{1}{z}}},\\
 |E_T^{r,H}(t,h,k;z)| &\ll_{r,T} k^{\frac{r}{2}}  |z|^{-r + \frac{1}{2}}  e^{- \frac{\pi}{k}\left(\frac{9}{4T}
-\frac{1}{12}\right)\Re{\frac{1}{z}}}.
\end{aligned}
\end{equation}
This implies the following result:
\begin{prop}\label{garvan_prop_moment_final}
Let $T>0$ be an odd integer and suppose that $h,k$, and $z$ satisfy
the usual conditions. Furthermore, suppose that $\Re{\frac{1}{z}}
\geq \frac{k}{2}$. Then, we have
\[ M_T^r\left(e^{\frac{2\pi
i}{k}(h+iz)}\right) =  M_T^{r,\mu}(h,k;z)+ E_T^{r,\mu}(h,k;z)+
\sum_{\substack{t=-\frac{T-1}{2}\\t \neq 0}}^{\frac{T-1}{2}}
\left(\sum_{l=0}^{\frac{k \gammaCo}{T}-1}
M_T^{r,H}(t,l,h,k;z)+E_T^{r,H}(t,l,h,k;z)\right),\] where the error
terms can be asymptotically bounded as in
\eqnref{garvan_eqn_moment_finalErrMu} and
\eqnref{garvan_eqn_moment_finalErrH}.
\end{prop}

\section{The Circle Method}
We are now in the position to apply the Circle Method. This is a
classical method and a detailed explanation can be found for example
in \cite{HaWr_int}. For our purpose we stick closely to the general
setup of \cite{BrMa_gen} (again a more detailed exposition can be
found also in the author's PhD thesis \cite{Wal_thesis}). We consider the $q$-series
\[ M_T^r(q)=\sum_{n=0}^\infty m_T^r(n) q^n.\] By Cauchy's Theorem,
we have
\[ m_T^r(n)=\frac{1}{2 \pi i}\int_\Gamma \frac{M_T^r(q)}{q^{n+1}} dq,\]
where $\Gamma$ is a simple counterclockwise oriented loop in the
unit circle around the origin. Choosing an explicit parametrization
of such a loop (depending on $n$), we have
\begin{align*} m_T^r(n)&= \int_0^1 M_T^r\left(e^{-\frac{2\pi}{n}+2 \pi i t}\right) e^{2 \pi-2 \pi i n
t}  dt.
\end{align*}
Now we set  $N:=\lfloor \sqrt{n} \rfloor$ and let $
\frac{h_1}{k_1} < \frac{h}{k} < \frac{h_2}{k_2}$ be adjacent Farey
fractions in the Farey sequence of order $N$. If we define
$\Gamma_{h,k}$ to be the arc parameterized by $t\mapsto
e^{-\frac{2\pi}{n}+2 \pi i t}$ for $t$ between $\frac{h_1+h}{k_1+k}$
and $\frac{h+h_2}{k+k_2}$, then the arcs $\Gamma_{h,k}$ make up the
entire circle. We find
\begin{equation}\label{garvan_eqn_integral_genformulaPhi} \int_0^1 M_T^r\left(e^{-\frac{2\pi}{n}+2 \pi i t}\right) e^{2
\pi-2 \pi i n t}  dt = \sum_{\substack{0 \leq h < k \leq N \\
(h,k)=1}} \int_{\Gamma_{h,k}} M_T^r\left(e^{-\frac{2\pi}{n}+2 \pi i
t}\right) e^{2 \pi-2 \pi i n t} dt.
\end{equation}
Substituting $t$ by $\phi+\frac{h}{k}$ in each integral over
$\Gamma_{h,k}$ and using the abbreviation $z=\frac{k}{n}-ik\phi$, we
may write \eqnref{garvan_eqn_integral_genformulaPhi} as
\begin{equation}\label{garvan_eqn_integral_genformulaZ}
\begin{aligned}
m_T^r(n) &= \sum_{\substack{0 \leq h < k \leq N\\ (h,k)=1}} e^{-\frac{2\pi i n h}{k}}\int_{-\vartheta_{h,k}'}^{\vartheta_{h,k}''} M_T^r\left(e^{\frac{2\pi i}{k}(h+iz)}\right)
e^{\frac{2\pi n z}{k}} d\phi,
\end{aligned}
\end{equation}
where  $-\vartheta_{h,k}':=-\frac{1}{k(k_1+k)}$ and
$\vartheta_{h,k}'':=\frac{1}{k(k_2+k)}$.

Now the variable $z$ is in the range where we can use the asymptotic
formulas for $M_T^r\left(e^{\frac{2\pi i}{k}(h+iz)}\right)$, coming
from Proposition \ref{garvan_prop_moment_final}. Hence, we can
evaluate \eqnref{garvan_eqn_integral_genformulaZ} by computing the
integrals
\[ \int_{-\vartheta'_{h,k}}^{\vartheta''_{h,k}}  M_T^{r,\mu}(h,k;z) e^{\frac{2\pi n z}{k}} d\phi, \qquad \int_{-\vartheta'_{h,k}}^{\vartheta''_{h,k}}  M_T^{r,H}(t,l,h,k;z) e^{\frac{2\pi n z}{k}} d\phi,
\]
and bounding the error integrals
\[ \int_{-\vartheta'_{h,k}}^{\vartheta''_{h,k}}
E_T^{r,\mu}(h,k;z) e^{\frac{2\pi n z}{k}} d\phi , \qquad
\int_{-\vartheta'_{h,k}}^{\vartheta''_{h,k}} E_T^{r,H}(t,h,k;z)
e^{\frac{2\pi n z}{k}} d\phi. \]
\subsection{The integral over $M_T^{r,\mu}$}
Ignoring the factors of $M_T^{r,\mu}(h,k;z)$ which do not depend on
$z$, we observe that in order to compute
\[ \int_{-\vartheta_{h,k}'}^{\vartheta_{h,k}''} M_T^{r,\mu}(h,k;z)
e^{\frac{2\pi n z}{k}} d\phi\] we are led to evaluate the integral
\begin{equation}\label{garvan_eqn_integral_simp}
\int_{-\vartheta_{h,k}'}^{\vartheta_{h,k}''} e^{-\frac{\pi
}{12k}\left(z-\frac{1}{z}\right) } z^{\frac{1}{2}-a-2c}
e^{\frac{2\pi n z}{k}} d\phi.
\end{equation}
In \cite{lehner_int}, Lehner has shown how to evaluate this integral.
We briefly outline Lehner's method, since we will proceed along
similar lines when computing the integral
involving $M_T^{r,H}$. Consider
\begin{equation}\label{garvan_eqn_integral_simp1}\int_{-\vartheta_{h,k}'}^{\vartheta_{h,k}''} z^{d} e^{\frac{\pi z}{k}(2n+\delta)+\frac{\pi}{kz}\beta}
d\phi,
\end{equation}
where $\delta$ and $\beta>0$ are rational numbers and $d \in
\frac{1}{2}+\Z$. We symmetrize the range of integration and return
to the variable $z$ through the change of variables
$\phi=\frac{iz}{k}-\frac{i}{n}$. This shows that
\eqnref{garvan_eqn_integral_simp1} is equal to
\[ \frac{i}{k} \int_{\frac{k}{n}+\frac{i}{N}}^{\frac{k}{n}-\frac{i}{N}}  z^{d} e^{\frac{\pi z}{k}(2n+\delta)+\frac{\pi}{kz}\beta}
d z = -\frac{i}{k}
\int^{\frac{k}{n}+\frac{i}{N}}_{\frac{k}{n}-\frac{i}{N}}  z^{d}
e^{\frac{\pi z}{k}(2n+\delta)+\frac{\pi}{kz}\beta} d z\] plus an
error term. After that, we again extend the range of integration to
a contour $\Gamma$ according to the following diagram in a
counterclockwise manner.
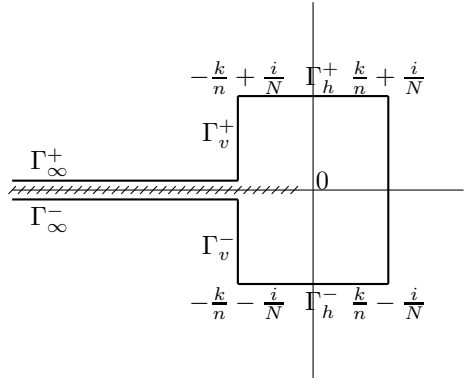
\begin{center}
\begin{tikzpicture}[scale=0.5,font=\small]\label{garvan_pic_contour}
\draw (-8,0) -- (4,0); \draw (0,-5) -- (0,5); 
\foreach \x in {-8.0,-7.75,...,-0.5} \draw (\x cm -0.1 cm,-3pt) --
(\x cm+0.1cm,3pt);
\draw[thick] (-8,0.25)--(-2,0.25);\draw[thick]
(-2,0.25)--(-2,2.5);\draw[thick] (-2,2.5)--(2,2.5);\draw[thick]
(2,2.5)--(2,-2.5); \draw[thick] (2,-2.5)--(-2,-2.5);\draw[thick]
(-2,-2.5)--(-2,-0.25); \draw[thick] (-8,-0.25)--(-2,-0.25);
\draw (-2,-3) node{$-\tfrac{k}{n}-\tfrac{i}{N}$}; \draw (2,-3)
node{$\tfrac{k}{n}-\tfrac{i}{N}$};\draw (-2,3)
node{$-\tfrac{k}{n}+\tfrac{i}{N}$};\draw (2,3)
node{$\tfrac{k}{n}+\tfrac{i}{N}$}; \draw (0.25,0.25) node {$0$}; 
\draw (-7,0.75) node{$\Gamma_\infty^+$};\draw (-7,-0.75)
node{$\Gamma_\infty^-$};\draw (-2.5,1.5) node{$\Gamma_v^+$};\draw
(-2.5,-1.5) node{$\Gamma_v^-$}; \draw (0.25,3) node{$\Gamma_h^+$};
\draw (0.25,-3) node{$\Gamma_h^-$};
\end{tikzpicture}
\end{center}
This again introduces an error term. It is possible to estimate all
occurring error terms and show that they are bounded by
$\frac{1}{k\sqrt{n}}(\tfrac{n}{k})^{|d|}$. Now the integral over the
contour in the picture above can be recognized as the
Schl{\"a}fli-representation of the modified Bessel function. Indeed,
according to \cite{wat_bes}, page 181, we have that
\begin{equation*}\label{garvan_eqn_integral_besseldef} I_\nu(x) = \frac{(\frac{x}{2})^\nu}{2 \pi
i}\int_{\Gamma} z^{-\nu-1} e^{z+\frac{x^2}{4z}}dz \end{equation*}
for any $x \in \R$ and $\nu \in \C$. Thus, the change of
variables $z \mapsto \frac{kz}{\pi(2n+\delta)}$ yields
\begin{lem}\label{garvan_cor_integral_intsimp}
Let $n,k$ be positive integers and $\delta$ and $\beta>0$ be rational
numbers and $d \in \frac{1}{2}+\Z$. Then,
\[ \int_{-\vartheta_{h,k}'}^{\vartheta_{h,k}''} z^{d} e^{\frac{\pi z}{k}(2n+\delta)+\frac{\pi}{kz}\beta}
d\phi = \frac{2 \pi}{ k}(2n+\delta)^{-\frac{d+1}{2}}
\beta^{\frac{d+1}{2}}   I_{-d-1}\left(\tfrac{2 \pi
\sqrt{\beta(2n+\delta)}}{k}\right)+
\frac{1}{k\sqrt{n}}(\tfrac{n}{k})^{|d|}.\]
\end{lem}
\subsection{The integral over $M_T^{r,H}$}
In order to evaluate
\[\int_{-\vartheta'_{h,k}}^{\vartheta''_{h,k}}  M_T^{r,H}(t,l,h,k;z)
e^{\frac{2\pi n z}{k}} d\phi,
\] we will pursue the same approach as for $M_T^{r,\mu}$, i.e., we will
modify the integral at the cost of introducing error terms until we can relate it to a modified Bessel function
(in fact to an integral over a modified Bessel function).

If we replaced $M_T^{r,H}(t,l,h,k;z)$ by the definition as in
\eqnref{garvan_eqn_moment_MTRmordell}, we would have to carry along
a lot of notation. Instead we first consider an integral of the
following general form
\begin{equation}\label{garvan_eqn_integral_Hgen} \int_{-\vartheta'_{h,k}}^{\vartheta''_{h,k}}    z^d  e^{\frac{\pi z}{k}(2n+\delta)} e^{\frac{\beta \pi}{k z}} \mathcal{H}_{c,T}(\alpha,\gamma,\varrho,k;z) d \phi.
\end{equation}
However, aiming at our later application we will put some restrictions on the parameters
that will later be satisfied automatically. We will assume that
$\alpha,\beta,\gamma,\delta,\varrho$ are rational
numbers, with $|\alpha|<\frac{1}{2}$ and $|\varrho|<\frac{1}{2}$, $a$ is a positive integer and
$c$ is an integer satisfying $c\leq r$ for some fixed integer $r$,
and $d \in \frac{1}{2}+\Z$. We will abbreviate these properties by
saying that $\alpha$, $\beta$, $\gamma$, $\delta$, $\varrho$, $c$,
and $d$ satisfy the ``usual conditions''.

Furthermore, in the application later the variables $\alpha,\beta,\gamma,\delta,\varrho$ can only attain
finitely many values for fixed $T$. Furthermore, for fixed $r$, also
the variables $c$ and $d$ can only attain finitely many values. This
implies that whenever there are error terms occurring which depend
on $\alpha,\beta,\gamma,\delta,\varrho,c$, or $d$, we can transform
these into error terms which may be bounded solely in terms of $T$
and $r$.

If one carries out our analysis in the context of \cite{BrMaRh_stat}, i.e. assuming that $T=1$ or $T=3$, then
the integrals in \eqnref{garvan_eqn_integral_Hgen} will only occur for $\beta \leq 0$. We first show that these
integrals are small and can be put into the error term. We give a full proof of this fact, and all will omit further proofs
in the following which proceed along the same lines.
\begin{lem}\label{garvan_lem_integral_betasmall}
Let $T>0$ be an odd integer and $r>0$ an even integer and suppose
that $\alpha$, $\beta$, $\gamma$, $\delta$, $\varrho$, $c$, $d$,
$h$, and $k$ satisfy the usual conditions. Further suppose that
$\beta \leq 0$. Then
\begin{equation*} \left|\int_{-\vartheta'_{h,k}}^{\vartheta''_{h,k}}    z^d  e^{\frac{\pi z}{k}(2n+\delta)} e^{\frac{\beta \pi}{k z}} \mathcal{H}_{c,T}(\alpha,\gamma,\varrho,k;z) d \phi \right| \ll_{r,T} \frac{1}{k \sqrt{n}}\left(\frac{n}{k}\right)^{|d|}.
\end{equation*}
\end{lem}
\begin{proof}
Taking absolute values, we see that \eqnref{garvan_eqn_integral_Hgen} is bounded by
\[  \int_{-\vartheta'_{h,k}}^{\vartheta''_{h,k}}    |z|^d e^{\frac{\pi
\Re{z}}{k}(2n+\delta)} e^{\frac{\beta \pi}{k } \Re{\frac{1}{z}}}
\left|\mathcal{H}_{c,T}(\alpha,\gamma,\varrho,k;z)\right| d \phi. \]
We symmetrize the range of integration using that $-\frac{1}{kN}
\leq -\vartheta'_{h,k}$ and $\vartheta''_{h,k} \leq \frac{1}{kN}$.
Then, we perform the change of variables
$\phi=\frac{iz}{k}-\frac{i}{n}$. Thus, we obtain that
\eqnref{garvan_eqn_integral_Hgen} is bounded by
\begin{equation}\label{garvan_eqn_integral_betasmall}  \frac{1}{k}
\int_{\frac{k}{n}+\frac{i}{N}}^{\frac{k}{n}-\frac{i}{N}}    |z|^d e^{\frac{\pi
\Re{z}}{k}(2n+\delta)} e^{\frac{\beta \pi}{k } \Re{\frac{1}{z}}}
|\mathcal{H}_{c,T}(\alpha,\gamma,\varrho,k;z)| dz.
\end{equation}
We now parameterize the range of integration by $z=
\frac{k}{n}-\frac{i}{N}\xi$ with $\xi \in [-1,1]$. We find that
$|z|=\sqrt{\frac{k^2}{n^2}+\frac{1}{N^2}\xi^2}$. Hence
\[ \frac{k^2}{n^2} \leq   |z|^2 \leq \frac{2}{n}. \]
Furthermore, we have $\Re{z} = \frac{k}{n}$ and, hence,
$\Re{\tfrac{1}{z}}= \frac{\Re{z}}{|z|^2} \geq \frac{k}{2}$. Using
this and $|\alpha|<\frac{1}{2}$, $|\varrho|<\frac{1}{2}$ we may show that
\[ |\mathcal{H}_{c,T}(\alpha,\gamma,\varrho,k;z)| \ll_{c,T} 1 \ll_{r,T} 1.\]
Plugging this into \eqnref{garvan_eqn_integral_betasmall} and using
$\beta \leq 0$, we see that \eqnref{garvan_eqn_integral_betasmall}
is essentially bounded (with an error term depending on $r$ and $T$)
by
\[  \frac{1}{k}
\int_{\frac{k}{n}+\frac{i}{N}}^{\frac{k}{n}-\frac{i}{N}}
\left|\frac{k}{n}\right|^d e^{\frac{\pi k}{kn}(2n+\delta)}
e^{\frac{\beta \pi}{k } \frac{n}{2k}} dz \ll_{r,T} \frac{1}{k
\sqrt{n}}\left(\frac{n}{k}\right)^{|d|}.
\]
\end{proof}

Now we turn to the case $\beta>0$. Our first goal is to replace the
term $e^{\frac{\beta \pi}{k z}}
\mathcal{H}_{c,T}(\alpha,\gamma,\varrho,k;z)$ in
\eqnref{garvan_eqn_integral_Hgen} by an expression which can be
integrated more easily. Of course this will be only possible by
introducing an extra error term.

\begin{prop}\label{garvan_prop_integral_PropSplit}
Let $T>0$ be an odd integer and $r>0$ an even integer and suppose
that $\alpha$, $\beta$, $\gamma$, $\delta$, $\varrho$, $c$, $d$,
$h$, and $k$ satisfy the usual conditions. Suppose further that
$\beta
> 0$, $z$ satisfies the usual conditions, and, additionally, that
$\Re{\frac{1}{z}} \geq \frac{k}{2}$. Then, we have
\[ e^{\frac{\beta \pi}{k z}} \mathcal{H}_{c,T}(\alpha,\gamma,\varrho,k;z)=\frac{\sqrt{\beta} \gamma
}{\sqrt{T}}\int_{-1}^1 \left(\tfrac{\sqrt{\beta} \gamma
}{\sqrt{T}}x+i\varrho\right)^c \frac{e^{-\frac{\pi \beta\left( x^2
-1\right)}{k z}+2 \pi \frac{\sqrt{\beta} \gamma }{\sqrt{T}} \alpha
x}}{\cosh\left(\pi \left(\tfrac{\sqrt{\beta} \gamma }{\sqrt{T}}x+i
\varrho\right)\right)} dx +
E_{\alpha,\beta,\gamma,\varrho,c,T}(z),\] where
$E_{\alpha,\beta,\gamma,\varrho,c,T}(z)$ is a function which is
bounded independently of $z$ by a constant $E_{r,T}$ depending only
on $r$ and $T$.
\end{prop}
\begin{proof}
We write
\[ e^{\frac{\beta \pi}{k z}} \mathcal{H}_{c,T}(\alpha,\gamma,\varrho,k;z)=e^{\frac{\beta \pi}{k z}}\int_{\R}
(x+i\varrho)^c \frac{e^{-\frac{\pi T x^2}{\gamma^2 k z}-2 \pi x
\alpha}}{\cosh(\pi (x+i \varrho))} dx.\] We see that the qualitative
behavior of the integrand (with respect to $z$) changes when the
sign in the exponent of $e^{\frac{\beta \pi}{k z}}e^{-\frac{\pi T
x^2}{\gamma^2 k z}}$ changes. This happens at
$\beta=-\frac{T}{\gamma^2}x^2$. This in turn leads us to make the
substitution $x \mapsto \frac{\sqrt{\beta} \gamma }{\sqrt{T}}x$,
which yields
\[ e^{\frac{\beta \pi}{k z}} \mathcal{H}_{c,T}(\alpha,\gamma,\varrho,k;z)=\frac{\sqrt{\beta} \gamma
}{\sqrt{T}}\int_{\R} \left(\tfrac{\sqrt{\beta} \gamma
}{\sqrt{T}}x+i\varrho\right)^c \frac{e^{-\frac{\pi \beta\left( x^2
-1\right)}{k z}+2 \pi \frac{\sqrt{\beta} \gamma }{\sqrt{T}} \alpha
x}}{\cosh\left(\pi \left(\tfrac{\sqrt{\beta} \gamma }{\sqrt{T}}x+i
\varrho\right)\right)} dx.\] We now split the integral into
$\int_\R=\int_{-\infty}^{-1}+\int_{-1}^{1}+ \int_1^\infty$. The
middle part becomes exactly the term in the statement of this
proposition, while the other two integrals form an error term.
Indeed, we have
\[ \left| \int_1^\infty \left(\tfrac{\sqrt{\beta} \gamma
}{\sqrt{T}}x+i\varrho\right)^c \frac{e^{-\frac{\pi \beta\left( x^2
-1\right)}{k z}+2 \pi \frac{\sqrt{\beta} \gamma }{\sqrt{T}} \alpha
x}}{\cosh\left(\pi \left(\tfrac{\sqrt{\beta} \gamma }{\sqrt{T}}x+i
\varrho\right)\right)} dx \right| \leq \int_1^\infty
\left|\tfrac{\sqrt{\beta} \gamma }{\sqrt{T}}x+i\varrho\right|^c
\frac{e^{2 \pi \frac{\sqrt{\beta} \gamma }{\sqrt{T}} \alpha
x}}{\left|\cosh\left(\pi \left(\tfrac{\sqrt{\beta} \gamma
}{\sqrt{T}}x+i \varrho\right)\right)\right|} dx, \] because
\[ \left| e^{-\frac{\pi \beta( x^2 -1)}{k
z}} \right| = e^{-\frac{\pi \beta( x^2 -1)}{k }\Re{\frac{1}{z}}}
\leq e^{-\frac{\pi \beta( x^2 -1)}{2 }} \leq 1.\] The remaining
integral is independent of $z$. Using a similar reasoning in the
case $\int_{-\infty}^{-1}$, we finish the proof of the proposition.
\end{proof}

Plugging the result of Proposition
\ref{garvan_prop_integral_PropSplit} into
\eqnref{garvan_eqn_integral_Hgen}, we obtain that
\begin{equation}\label{garvan_eqn_integral_split}
\begin{aligned}
&\int_{-\vartheta'_{h,k}}^{\vartheta''_{h,k}}    z^d  e^{\frac{\pi z}{k}(2n+\delta)} e^{\frac{\beta \pi}{k z}} \mathcal{H}_{c,T}(\alpha,\gamma,\varrho,k;z) d \phi \\
&=\int_{-\vartheta'_{h,k}}^{\vartheta''_{h,k}}    z^d  e^{\frac{\pi
z}{k}(2n+\delta)}\left( \tfrac{\sqrt{\beta} \gamma
}{\sqrt{T}}\int_{-1}^1 \left(\tfrac{\sqrt{\beta} \gamma
}{\sqrt{T}}x+i\varrho\right)^c \frac{e^{-\frac{\pi \beta\left( x^2
-1\right)}{k z}+2 \pi \frac{\sqrt{\beta} \gamma }{\sqrt{T}} \alpha
x}}{\cosh\left(\pi \left(\tfrac{\sqrt{\beta} \gamma }{\sqrt{T}}x+i
\varrho\right)\right)} dx + E_{\alpha,\beta,\gamma,\rho,c,T}(z)
\right) d \phi.
\end{aligned}
\end{equation}
We split this integral and evaluate the main term and the error term
separately. Proceeding analogously to Lemma \ref{garvan_lem_integral_betasmall} one
proves that the contribution from the error term is indeed small:
\begin{lem}\label{garvan_lem_integral_Herr}
Let $T>0$ be an odd integer and $r>0$ an even integer and suppose
that $\alpha$, $\beta$, $\gamma$, $\delta$, $\varrho$, $c$, $d$,
$h$, and $k$ satisfy the usual conditions and suppose $\beta>0$.
Then,
\[ \left|\int_{-\vartheta'_{h,k}}^{\vartheta''_{h,k}} z^d  e^{\frac{\pi z}{k}(2n+\delta)}
E_{\alpha,\beta,\gamma,\varrho,c,T}(z) d \phi \right| \ll_{r,T}
\frac{1}{k \sqrt{n}}\left(\frac{n}{k}\right)^{|d|}.\]
\end{lem}
We next consider
\begin{equation}\label{garvan_eqn_integral_main1} \int_{-\vartheta'_{h,k}}^{\vartheta''_{h,k}}    z^d  e^{\frac{\pi z}{k}(2n+\delta)}\left( \tfrac{\sqrt{\beta} \gamma
}{\sqrt{T}}\int_{-1}^1 \left(\tfrac{\sqrt{\beta} \gamma
}{\sqrt{T}}x+i\varrho\right)^c \frac{e^{-\frac{\pi \beta\left( x^2
-1\right)}{k z}+2 \pi \frac{\sqrt{\beta} \gamma }{\sqrt{T}} \alpha
x}}{\cosh\left(\pi \left(\tfrac{\sqrt{\beta} \gamma }{\sqrt{T}}x+i
\varrho\right)\right)} dx \right) d \phi,
\end{equation}
which is the main part of \eqnref{garvan_eqn_integral_split}. We
first rewrite \eqnref{garvan_eqn_integral_main1} as
\begin{equation}\label{garvan_eqn_integral_main2} \lambda\int_{-\vartheta'_{h,k}}^{\vartheta''_{h,k}}    z^d  e^{\frac{\pi z}{k}(2n+\delta)} \int_{-1}^1 \left(\lambda x+i\varrho\right)^c \frac{e^{-\frac{\pi \beta\left( x^2 -1\right)}{k
z}-2 \pi \lambda \alpha x}}{\cosh\left(\pi \left(\lambda x+i
\varrho\right)\right)} dx  d \phi,
\end{equation} where
$\lambda:=\frac{\sqrt{\beta} \gamma }{\sqrt{T}}$. In order to evaluate this integral we
combine Lehner's approach \cite{lehner_int}, which we sketched in the previous
section, with work of Bringmann
\cite{Br_def} and Bringmann and Mahlburg \cite{BrMa_gen} where similar integrals were
considered.
First, we symmetrize the range of integral in \eqnref{garvan_eqn_integral_main2}:
\[ \int_{-\vartheta'_{h,k}}^{\vartheta''_{h,k}} = \int_{-\frac{1}{kN}}^{\frac{1}{kN}}-\int_{\vartheta''_{h,k}}^{\frac{1}{kN}} - \int_{-\frac{1}{kN}}^{-\vartheta'_{h,k}}.\]
Again it is rather easy to bound the second and the third integral and see that their contribution is small and will form part of the error term
later.
\begin{lem}\label{garvan_lem_integral_Hmainerr}
Let $T>0$ be an odd integer and $r>0$ an even integer and suppose
that $\alpha$, $\beta$, $\gamma$, $\delta$, $\varrho$, $c$, $d$,
$h$, and $k$ satisfy the usual conditions. Suppose that $\beta>0$.
Then, we have
\[ \lambda\int_{\vartheta''_{h,k}}^{\frac{1}{kN}} \int_{-1}^1   z^d  e^{\frac{\pi z}{k}(2n+\delta)}  \left(\lambda x+i\varrho\right)^c \frac{e^{-\frac{\pi \beta\left( x^2 -1\right)}{k
z}+2 \pi \lambda \alpha x}}{\cosh\left(\pi \left(\lambda x+i
\varrho\right)\right)} dx  d \phi \ll_{r,T}
\frac{1}{k\sqrt{n}}\left(\frac{n}{k}\right)^{|d|}\] and
\[ \lambda\int_{-\frac{1}{kN}}^{-\vartheta'_{h,k}} \int_{-1}^1   z^d  e^{\frac{\pi z}{k}(2n+\delta)}  \left(\lambda x+i\varrho\right)^c \frac{e^{-\frac{\pi \beta\left( x^2 -1\right)}{k
z}+2 \pi \lambda \alpha x}}{\cosh\left(\pi \left(\lambda x+i
\varrho\right)\right)} dx  d \phi \ll_{r,T}
\frac{1}{k\sqrt{n}}\left(\frac{n}{k}\right)^{|d|}.\]
\end{lem}

Now we consider the main term contribution of
\eqnref{garvan_eqn_integral_main2} after symmetrization. We rewrite
this integral using $\phi=\frac{iz}{k}-\frac{i}{n}$.
\begin{equation}\label{garvan_eqn_integral_lehnerprep}
\begin{aligned}
&\lambda \int_{-\frac{1}{kN}}^{\frac{1}{kN}} \int_{-1}^1   z^d
e^{\frac{\pi z}{k}(2n+\delta)}  \left(\lambda x+i\varrho\right)^c
\frac{e^{-\frac{\pi \beta\left( x^2 -1\right)}{k z}+2 \pi \lambda
\alpha x}}{\cosh\left(\pi \left(\lambda x+i \varrho\right)\right)}
dx  d
\phi \\
&= \lambda \int_{-1}^1 \frac{\left(\lambda x+i\varrho\right)^c e^{+2
\pi \lambda \alpha x}}{\cosh\left(\pi \left(\lambda x+i
\varrho\right)\right)}
\left(-\frac{i}{k}\int^{\frac{k}{n}+\frac{i}{N}}_{\frac{k}{n}-\frac{i}{N}}
z^d e^{\frac{\pi z}{k}(2n+\delta)} e^{\frac{\pi \beta\left(1- x^2
\right)}{k z}}
  dz \right) dx.
\end{aligned}
\end{equation}
Our first goal is to compute the inner integral on the right hand
side of \eqnref{garvan_eqn_integral_lehnerprep}. In order to do that
we extend the range of integration as given in the diagram on page
\pageref{garvan_pic_contour}. Note that we have to slit the plane
because $z^d$ is a multi-valued function. On this slitted plane we
use the principal branch of the logarithm. We now prove that the
contribution from all the additional paths is small.
\subsubsection{The contours $\Gamma_h^+$ and $\Gamma_h^-$}

\begin{lem}\label{garvan_lem_integral_contH}
Uniformly for all $x \in [-1,1]$, we have
\[ \left|-\frac{i}{k}\int_{\Gamma_h^+} z^d e^{\frac{\pi z}{k}(2n+\delta)}
e^{\frac{\pi \beta\left(1- x^2\right)}{k z}}\right| \ll_T 1. \] The
same is true for the integral over $\Gamma_h^-$.
\end{lem}
\begin{proof}
In the range of integration one has $|z|^d \leq n^{-\frac{|d|}{2}}$, $\Re{z} \leq \frac{k}{n}$ and
$\Re{\frac{1}{z}}=\frac{\Re{z}}{ |z|^2} \leq k$. We conclude that
\begin{align*} \left| \int_{\Gamma_h^+} z^d e^{\frac{\pi z}{k}(2n+\delta)}
e^{\frac{\pi \beta\left(1- x^2\right)}{k z}} dz\right|  &\leq
\int_{\Gamma_h^+} |z|^d e^{\frac{\pi }{k}(2n+\delta) \Re{z}}
e^{\frac{\pi \beta\left(1- x^2\right)}{k}\Re{\frac{1}{z}}} dz  \leq
\int_{\Gamma_h^+} n^{-\frac{|d|}{2}} e^{ \pi
\left(2+\frac{\delta}{n}\right) } e^{\pi \beta\left(1- x^2\right)}
dz.
\end{align*}
This integral can be bounded by a constant only depending on $T$
uniformly in $x$.
\end{proof}
\subsubsection{The contours $\Gamma_v^+$ and $\Gamma_v^-$}
\begin{lem}\label{garvan_lem_integral_contV}
Uniformly for all $x \in [-1,1]$, we have
\[ \left|-\frac{i}{k}\int_{\Gamma_v^+} z^d e^{\frac{\pi z}{k}(2n+\delta)}
e^{\frac{\pi \beta\left(1- x^2\right)}{k z}} dz\right|  \ll_{d,T}
\frac{1}{k\sqrt{n}}\left(\frac{n}{k}\right)^{|d|}. \] The same
estimate holds for the integral over $\Gamma_v^-$.
\end{lem}
\begin{proof}
Here, in the range of integration, we have $|z|^d \leq \left(\frac{n}{k}\right)^{|d|}$.
Of course, the real parts of both $z$ and $\frac{1}{z}$ are negative
on $\Gamma_v^+$. Hence, we obtain
\[\left|-\frac{i}{k}
\int_{\Gamma_v^+} z^d e^{\frac{\pi z}{k}(2n+\delta)} e^{\frac{\pi
\beta\left(1- x^2\right)}{k z}} dz \right| \leq
\frac{1}{k}\int_{\Gamma_v^+} |z|^d \underbrace{e^{\frac{\pi
}{k}(2n+\delta) \Re{z}}}_{\leq 1} \underbrace{e^{\frac{\pi
\beta\left(1- x^2\right)}{k } \Re{\frac{1}{z}}}}_{\leq 1} dz
\ll_{d,T} \frac{1}{k\sqrt{n}}\left(\frac{n}{k}\right)^{|d|}.
\]
\end{proof}

\subsubsection{The contours $\Gamma_\infty^+$ and $\Gamma_\infty^-$}

\begin{lem}\label{garvan_lem_integral_contInfty}
Uniformly for all $x \in [-1,1]$, we have
\[ \left|-\frac{i}{k}\int_{\Gamma_\infty^+} z^d e^{\frac{\pi z}{k}(2n+\delta)}
e^{\frac{\pi \beta\left(1- x^2\right)}{k z}} dz \right| \ll_{d,T}
\frac{1}{k\sqrt{n}}\left(\frac{n}{k}\right)^{|d|}.
\]
The same estimate holds for the integral over $\Gamma_\infty^-$.
\end{lem}
\begin{proof}
We find
\[ \int_{\Gamma_\infty^+} z^d e^{\frac{\pi z}{k}(2n+\delta)}
e^{\frac{\pi \beta\left(1- x^2\right)}{k z}} dz=
\int_{-\frac{k}{N}}^{-\infty} e^{d \log z} e^{\frac{\pi
z}{k}(2n+\delta)} e^{\frac{\pi \beta(1- x^2)}{k z}} dz,\] and,
similarly,
\[ \int_{\Gamma_\infty^-} z^d e^{\frac{\pi z}{k}(2n+\delta)}
e^{\frac{\pi \beta(1- x^2)}{k z}} dz= e^{-2 \pi i
d}\int^{-\frac{k}{N}}_{-\infty} e^{d (\log z - 2\pi  i) }
e^{\frac{\pi z}{k}(2n+\delta)} e^{\frac{\pi \beta\left(1-
x^2\right)}{k z}} dz.\] In either case it remains to bound
\begin{equation}\label{garvan_eqn_integral_lehner1} \int_{-\infty}^{-\frac{k}{N}} e^{d \log z}  e^{\frac{\pi z}{k}(2n+\delta)}
e^{\frac{\pi \beta\left(1- x^2\right)}{k z}} dz. \end{equation}
Substituting $z \mapsto -\frac{1}{z}$, we find that
\eqnref{garvan_eqn_integral_lehner1} is bounded by
\[  \int_{0}^{\frac{N}{k}} |z|^{-d-2} e^{-\frac{\pi }{k}(2n+\delta)\Re{\frac{1}{z}}}
e^{\frac{\pi \beta\left( x^2 -1\right)}{k}\Re{z}} dz =
\int_{0}^{\frac{N}{k}} |z|^{-d-2} e^{-\frac{\pi }{kz}(2n+\delta)}
e^{\frac{\pi \beta\left( x^2 -1\right)z}{k}} dz.\] Using that $x \in
[-1,1]$ and that $z \in \R^+$ in the range of integration, we see
that $e^{\frac{\pi \beta\left( x^2 -1\right)z}{k }} \leq 1$. Thus,
\eqnref{garvan_eqn_integral_lehner1} is bounded by
\[  \int_{0}^{\frac{N}{k}} |z|^{-d-2} e^{-\frac{\pi }{kz}(2n+\delta)}
dz. \] Now we distinguish two cases. If $-d-2<0$, then we see that
the function
\[ z \mapsto |z|^{-d-2} e^{-\frac{\pi }{kz}(1+\delta)} \]
is bounded for $z \in \R^+$ by a constant depending on $d$.
Furthermore, the function $z \mapsto e^{-\frac{\pi }{zk}(2n-1)}$ is
monotonically increasing in $z$. Thus, it attains its maximal value
at the boundary at $\frac{N}{k}$, namely $e^{-\frac{\pi }{N}(2n-1)}
\leq 1$. We may now assume that $-d-2\geq0$. In this case the
function $z \mapsto |z|^{-d-2} e^{-\frac{\pi }{kz}(2n+\delta)}$ is
monotonically increasing and we may again evaluate at $\frac{N}{k}$,
in which case the integral is bounded by
\[ \left(\frac{N}{k}\right)^{-d-2}=\frac{1}{\sqrt{n}}\left(\frac{n}{k}\right)^{|d|}n^{-|d|/2-\frac{1}{2}}k^2.\]
Using the fact that $|d| \geq 2$, completes the proof of the lemma.
\end{proof}

\subsubsection{Schl{\"a}fli's integral and the evaluation}

Combining Lemmas \ref{garvan_lem_integral_contV},
\ref{garvan_lem_integral_contH}, and
\ref{garvan_lem_integral_contInfty}, we see that
\[ -\frac{i}{k}\int^{\frac{k}{n}+\frac{i}{N}}_{\frac{k}{n}-\frac{i}{N}} z^d e^{\frac{\pi z}{k}(2n+\delta)}
e^{\frac{\pi \beta\left(1- x^2\right)}{k z}} dz =
-\frac{i}{k}\int_\Gamma z^d e^{\frac{\pi z}{k}(2n+\delta)}
e^{\frac{\pi \beta\left(1- x^2\right)}{k z}} dz +
E_{T,\delta,\beta}(d,k,n;x),\] where the error term satisfies
\[ |E_{T,\delta,\beta}(d,k,n;x)| \ll_{d,T} \frac{1}{k\sqrt{n}}\left(\frac{n}{k}\right)^{|d|}\]
uniformly for all $x \in [-1,1]$. As in Lemma
\ref{garvan_cor_integral_intsimp}, we use the Schl{\"a}fi integral
representation given on page 181 of \cite{wat_bes} to see that
\[ -\frac{i}{k}\int_\Gamma z^d e^{\frac{\pi z}{k}(2n+\delta)}
e^{\frac{\pi \beta\left(1- x^2\right)}{k z}} dz= \frac{2 \pi}{
k}(2n+\delta)^{-\frac{d+1}{2}} \beta^{\frac{d+1}{2}}\left(1-
x^2\right)^{\frac{d+1}{2}} I_{-d-1}\left(\tfrac{2 \pi
\sqrt{\beta\left(1- x^2\right)(2n+\delta)}}{k}\right). \] Inserting
this into \eqnref{garvan_eqn_integral_lehnerprep}, we obtain
\begin{align*} &\lambda \int_{-\frac{1}{kN}}^{\frac{1}{kN}}
\int_{-1}^1 z^d e^{\frac{\pi z}{k}(2n+\delta)}  \left(\lambda
x+i\varrho\right)^c \frac{e^{-\frac{\pi \beta\left( x^2 -1\right)}{k
z}-2 \pi \lambda \alpha x}}{\cosh\left(\pi \left(\lambda x+i
\varrho\right)\right)} dx  d \phi\\
& = \lambda \int_{-1}^1 \frac{\left(\lambda x+i\varrho\right)^c e^{2
\pi \lambda \alpha x}}{\cosh\left(\pi \left(\lambda x+i
\varrho\right)\right)} \frac{2 \pi}{ k}(2n+\delta)^{-\frac{d+1}{2}}
\beta^{\frac{d+1}{2}}(1-x^2)^{\frac{d+1}{2}} I_{-d-1}\left(\tfrac{2
\pi \sqrt{\beta\left(1- x^2\right)(2n+\delta)}}{k}\right)dx \\
& \phantom{=} + \lambda\int_{-1}^1 \frac{\left(\lambda
x+i\varrho\right)^c e^{2 \pi \lambda \alpha x}}{\cosh\left(\pi
\left(\lambda x+i \varrho\right)\right)}E_{T,\delta,\beta}(d,k,n;x)
dx.
\end{align*}
Finally, one shows that the last summand is again of the size of the
error terms, which we obtained so far. This allows us now to give
the final formula for \eqnref{garvan_eqn_integral_Hgen}. In
order to state our result more succinctly, we introduce the
abbreviation
\begin{equation}\label{garvan_eqn_integral_besselint}
\mathcal{I}_{T;\alpha,\beta,\delta,\varrho}(c,d,k;n):=\int_{-1}^1
\frac{\left(\frac{\sqrt{\beta}\gammaCo}{\sqrt{T}}
x+i\varrho\right)^c e^{2 \pi \frac{\sqrt{\beta}\gammaCo}{\sqrt{T}}
\alpha x}}{\cosh\left(\pi
\left(\frac{\sqrt{\beta}\gammaCo}{\sqrt{T}} x+i
\varrho\right)\right)} \left(1- x^2\right)^{\frac{d+1}{2}}
I_{-d-1}\left(\tfrac{2 \pi \sqrt{\beta\left(1-
x^2\right)(2n+\delta)}}{k}\right)dx.
\end{equation}

Then, combining Lemma \ref{garvan_lem_integral_Herr},
\ref{garvan_lem_integral_Hmainerr}, \ref{garvan_lem_integral_contH}, \ref{garvan_lem_integral_contV} and \ref{garvan_lem_integral_contInfty} we obtain the following expression for \eqnref{garvan_eqn_integral_Hgen}:
\begin{prop}\label{garvan_prop_integral_PPintegralFinal}
Let $T>0$ be an odd integer, $r>0$ an even integer and suppose that
$\alpha$, $\beta$, $\gamma$, $\delta$, $\varrho$, $c$, $d$, $h$, and
$k$ satisfy the usual conditions. Further suppose that $\beta>0$.
Then
\[ \int_{-\vartheta'_{h,k}}^{\vartheta''_{h,k}} \hspace*{-1.0em}z^d  e^{\frac{\pi z}{k}(2n+\delta)} e^{\frac{\beta \pi}{k z}} \mathcal{H}_{c,T}(\alpha,\gamma,\varrho,k;z) d \phi = \frac{2 \pi\gamma}{k\sqrt{T}} (2n+\delta)^{-\frac{d+1}{2}}
\beta^{\frac{d}{2}+1}
\mathcal{I}_{T;\alpha,\beta,\delta,\varrho}(c,d,k;n) +
O_{r,T}\left(\tfrac{1}{k
\sqrt{n}}\left(\tfrac{n}{k}\right)^{|d|}\right).\]
\end{prop}
\subsection{Integrating the errors}
Our last task in this section will be to bound the integrals
\[ \int_{-\vartheta'_{h,k}}^{\vartheta''_{h,k}} E_T^{r,H}(t,h,k;z) e^{\frac{2\pi n z}{k}} d\phi \qquad \text{and}  \qquad \int_{-\vartheta'_{h,k}}^{\vartheta''_{h,k}}
E_T^{r,\mu}(h,k;z) e^{\frac{2\pi n z}{k}} d\phi. \] We see that at
this point we have to put a restriction on $T$.
\begin{lem}\label{garvan_lem_integral_errorH}
Let $h,k$ satisfy the usual conditions. Suppose that $0<T<27$ is an
odd integer. Then,
\[ \left|\int_{-\vartheta'_{h,k}}^{\vartheta''_{h,k}} E_T^{2,H}(t,l,h,k;z) e^{\frac{2\pi n z}{k}} d\phi \right| \ll_{T} n k^{-2} \ \ \text{and} \ \ \left|\int_{-\vartheta'_{h,k}}^{\vartheta''_{h,k}} E_T^{r,H}(t,l,h,k;z) e^{\frac{2\pi n z}{k}} d\phi \right| \ll_{r,T} n^{r-1} k^{-\frac{r}{2}-\frac{1}{2}}. \]
\end{lem}
\begin{proof}
We only treat the case $r>2$, and $r=2$ is proven similarly. Using
Proposition \ref{garvan_prop_moment_final}, taking absolute values and extending the range of integration,
we see that
\[ \left|\int_{-\vartheta'_{h,k}}^{\vartheta''_{h,k}} E_T^{r,H}(t,l,h,k;z) e^{\frac{2\pi n z}{k}} d\phi. \right|
\ll_{r,T} \frac{1}{k}
\int_{\frac{k}{n}+\frac{i}{N}}^{\frac{k}{n}-\frac{i}{N}}
k^{\frac{r}{2}}  |z|^{-r + \frac{1}{2}} e^{-
\frac{\pi}{k}\left(\frac{9}{4T}
-\frac{1}{12}\right)\Re{\frac{1}{z}}} dz. \] As in Lemma
\ref{garvan_lem_integral_betasmall}, we observe that in the given
range of integration we have $\Re{z}=\frac{k}{n}$ and $\Re{\tfrac{1}{z}}=
\frac{\Re{z}}{|z|^2} \geq \frac{k}{2}$. If $0< T < 27$, the sign in
the exponent is negative. Hence, we obtain
\[ \left|\int_{-\vartheta'_{h,k}}^{\vartheta''_{h,k}} E_T^{r,H}(t,l,h,k;z) e^{\frac{2\pi n z}{k}} d\phi. \right|
\ll_{r,T} \frac{1}{k}
\int_{\frac{k}{n}+\frac{i}{N}}^{\frac{k}{n}-\frac{i}{N}}
k^{\frac{r}{2}}  \left|\tfrac{k}{n}\right|^{-r + \frac{1}{2}} e^{-
\frac{\pi}{k}\left(\frac{9}{4T}
-\frac{1}{12}\right)\frac{k}{2}} dz \ll_{r,T} n^{r-1}
k^{-\frac{r}{2}-\frac{1}{2}}.
\]
\end{proof}
\begin{lem}\label{garvan_lem_integral_errorMu}
Let $h,k$ satisfy the usual conditions. Suppose that $0<T<24$ is an
odd integer. Then,
\[ \left|\int_{-\vartheta'_{h,k}}^{\vartheta''_{h,k}}E_T^{2,\mu}(h,k;z) e^{\frac{2\pi n z}{k}} d\phi\right| \ll_{T} n k^{-\frac{5}{2}} \qquad \text{and} \qquad \left|\int_{-\vartheta'_{h,k}}^{\vartheta''_{h,k}} E_T^{r,\mu}(h,k;z) e^{\frac{2\pi n z}{k}} d\phi\right| \ll_{r,T} k^{-\frac{r}{2}-\frac{1}{2}}.\]
\end{lem}

\section{Proof of the main theorems}
\subsection{Proof of Theorem A}
Recall that \eqnref{garvan_eqn_integral_genformulaZ} states that
\[ m_T^r(n) = \sum_{\substack{0 \leq h < k \leq N\\ (h,k)=1}} e^{-\frac{2\pi i n h}{k}}\int_{-\vartheta_{h,k}'}^{\vartheta_{h,k}''} M_T^r\left(e^{\frac{2\pi i}{k}(h+iz)}\right)
e^{\frac{2\pi n z}{k}} d\phi. \] Furthermore, recall the
representation
\[M_T^r\left(e^{\frac{2\pi i}{k}(h+iz)}\right)= M_T^{r,\mu}(h,k;z)+ E_T^{r,\mu}(h,k;z) + \sum_{\substack{t=-\frac{T-1}{2}\\t \neq 0}}^{\frac{T-1}{2}} \left(\sum_{l=0}^{\frac{k \gammaCo}{T}-1} M_T^{r,H}(t,l,h,k;z)+E_T^{r,H}(t,l,h,k;z)\right) \] given in Proposition \ref{garvan_prop_moment_final}.
Combining these equations with the integral evaluations and
estimates from the previous section, we now derive the asymptotic
formula for $m_T^r(n)$ as given in Theorem A. We treat the
contributions coming from $M_T^{r,\mu}$, $M_T^{r,H}$, $E_T^{r,\mu}$,
and $E_T^{r,H}$ separately.

To describe the contribution coming from $M_T^{r,\mu}$, we require
the Kloosterman sum
\begin{equation}\label{garvan_eqn_circle_klooster1} K_k(n):=-i^{\frac{3}{2}} \sum_{\substack{0 \leq h < k\\ (h,k)=1}} e^{-\frac{2\pi i n
h}{k}} e^{\frac{\pi i}{12k}\left(h-[h]_k\right)
}\chi^{-1}\left(h,[-h]_k,k\right).
\end{equation}

\begin{prop}\label{garvan_prop_circle_MuCont}
We have
\begin{align*} &\sum_{\substack{0 \leq h < k \leq N\\ (h,k)=1}} e^{-\frac{2\pi
i n h}{k}}\int_{-\vartheta_{h,k}'}^{\vartheta_{h,k}''}
M_T^{r,\mu}(h,k;z) e^{\frac{2\pi n z}{k}} d\phi\\
&= 2 \pi \sum_{k \leq \sqrt{n}} \frac{K_k(n)}{k} \sum_{2a+2b+2c=r}
\kappa(a,b,c) (kT)^a (24n-1)^{-\frac{3}{4}+\frac{a}{2}+c}
I_{-\frac{3}{2}+a+2c}\left(\tfrac{\pi \sqrt{24n-1}}{6k}\right) +
O_{r,T}\left(n^{r-1}\right).
\end{align*}
\end{prop}
\begin{proof}
Using \eqnref{garvan_eqn_moment_MTRmu}, we see that the contribution
coming from $M_T^{r,\mu}$ is given by
\begin{align*} &\sum_{\substack{0 \leq h < k \leq N\\ (h,k)=1}} e^{-\frac{2\pi i n h}{k}} \int_{-\vartheta_{h,k}'}^{\vartheta_{h,k}''} M_T^{r,\mu}(h,k;z)
e^{\frac{2\pi n z}{k}} d\phi = -i^{\frac{3}{2}}\sum_{\substack{0 \leq h < k \leq N\\ (h,k)=1}} e^{-\frac{2\pi i n h}{k}} e^{\frac{\pi i}{12k}\left(h-[h]_k\right) }\chi^{-1}\left(h,[-h]_k,k\right)\\
& \times \sum_{2a+2b+2c=r} \kappa(a,b,c) (kT)^a
\int_{-\vartheta_{h,k}'}^{\vartheta_{h,k}''} e^{-\frac{\pi
}{12k}\left(z-\frac{1}{z}\right) } z^{\frac{1}{2}-a-2c}
e^{\frac{2\pi n z}{k}} d\phi.
\end{align*}

Using Lemma \ref{garvan_cor_integral_intsimp} and the definition of
the Kloosterman sum, we immediately obtain the main term. Hence, we
are left with estimating the error term. After taking absolute
values and using Lemma \ref{garvan_cor_integral_intsimp}, we see
that the error term is essentially bounded (in terms of $r$ and $T$)
by
\begin{align*} \sum_{\substack{0 \leq h < k \leq N\\
(h,k)=1}} &\sum_{2a+2b+2c=r} \kappa(a,b,c) (kT)^a
\frac{1}{k\sqrt{n}}\left(\frac{n}{k}\right)^{\left|\frac{1}{2}-a-2c\right|}
\ll_{r,T} \sum_{k \leq N} \sum_{2a+2b+2c=r}
n^{\left|\frac{1}{2}-a-2c\right|-\frac{1}{2}}
k^{a-\left|\frac{1}{2}-a-2c\right|}.
\end{align*}
We distinguish the contributions to the error term coming from two
subcases. Either $a=b=0$ and $2c=r$, and, hence,
\[ n^{\left|\frac{1}{2}-a-2c\right|-\frac{1}{2}} k^{a-\left|\frac{1}{2}-a-2c\right|}=n^{r-1} k^{-\left|\frac{1}{2}-r\right|}.\]
In this case, the contribution to the error term is
$O_{r,T}(n^{r-1})$. In the other case we find
\[ n^{\left|\frac{1}{2}-a-2c\right|-\frac{1}{2}} k^{a-\left|\frac{1}{2}-a-2c\right|} \leq n^{r-2} k^\frac{1}{2}. \]
Hence, the contribution to the error term is of size
\[ \sum_{k \leq N} n^{\left|\frac{1}{2}-a-2c\right|-\frac{1}{2}} k^{a-\left|\frac{1}{2}-a-2c\right|}
 \ll_{r,T} \sum_{k \leq N} n^{r-2} k^\frac{1}{2} \ll n^{r-2}N^\frac{3}{2} \ll n^{r-1}.\]
\end{proof}

We next consider
\[ \sum_{\substack{0 \leq h < k \leq N\\ (h,k)=1}} e^{-\frac{2\pi i n h}{k}}\int_{-\vartheta_{h,k}'}^{\vartheta_{h,k}''} \sum_{\substack{t=-\frac{T-1}{2} \\ t \neq 0}}^{\frac{T-1}{2}} \sum_{l=0}^{\frac{k \gammaCo}{T}-1}
M_T^{r,H}(t,l,h,k;z) e^{\frac{2\pi n z}{k}} d\phi,\] which is the
contribution for $m_T^r(n)$ coming from $M_T^{r,H}$. We define the
partial Kloosterman sum
\begin{equation}\label{garvan_eqn_circle_klooster2}
K_{\sigma,\varrho,l;k}(n):=\sum_{\substack{0 \leq h < k\\ (h,k)=1 \\
\rhoT\left(\frac{tTh}{\gammaGCD}\right)=\varrho}} e^{-\frac{2\pi i n
h}{k}} U_H^{\ast}(T,t,l,h,k).
\end{equation}
Now we proceed as before in the case of $M_T^{r,\mu}$ and use the integral
evaluations from the previous section and bound the occurring error terms.
Then, our result can be stated as follows:
\begin{prop}\label{garvan_prop_circle_Hcont}
We find that
\begin{align*}
&\sum_{\substack{0 \leq h < k \leq N\\ (h,k)=1}} e^{-\frac{2\pi i n
h}{k}}\int_{-\vartheta_{h,k}'}^{\vartheta_{h,k}''}
\sum_{t=-\frac{T-1}{2},t \neq 0}^{\frac{T-1}{2}} \sum_{l=0}^{\frac{k
\gammaCo}{T}-1} M_T^{r,H}(t,l,h,k;z) e^{\frac{2\pi n z}{k}} d\phi \\
&= 2 \pi \sum_{\gammaGCD|T} \sum_{\substack{t=-\frac{T-1}{2} \\ t
\neq 0}}^{\frac{T-1}{2}}
\sum_{\varrho=-\frac{T-1}{2}}^{\frac{T-1}{2}}
 \sum_{\substack{0<k \leq N\\(k,T)=\gammaGCD}}
\sum_{l=0}^{\frac{k} {\gammaGCD}-1}
\frac{K_{\sigma,\varrho,l;k}(n)}{k} \sum_{2a+(2b+1)+c=r}
\kappaH(a,b,c) k^{b-\frac{1}{2}} T^{b-\frac{1}{2}}
\gammaGCD^{c+\frac{1}{2}}
\left(2n-\tfrac{1}{12}\right)^{\frac{a+c}{2}-\frac{1}{4}}\\
& \times \hspace*{-0.25em}
\left(\tfrac{1}{12}-\tfrac{\gammaGCD^2}{T^3}\left(\varrho^2+\tfrac{T^2}{4}-|\varrho|T\right)\right)_+^{\frac{3}{4}-\frac{a+c}{2}}
\mathcal{I}_{T;\alpha_{T,t}(l,\frac{k}{\gammaGCD}),\frac{1}{12}-\frac{\gammaGCD^2}{T^3}\left(\varrho^2+\frac{T^2}{4}-|\varrho|T\right),-\frac{1}{12},\frac{\varrho}{T}}\left(c,-\tfrac{1}{2}-a-c,k;n\right)
\hspace*{-0.25em}+\hspace*{-0.25em}E_{r,T}(n).
\end{align*}
Here, $E_{r,T}(n)$ is an error term. If $r=2$, then the error term
has the magnitude $O_{T}\left(n \log n\right)$, whereas it is of
order $O_{r,T}\left(n^{r-1}\right)$ if $r>2$.
\end{prop}

Finally, we use Lemmas \ref{garvan_lem_integral_errorH} and
\ref{garvan_lem_integral_errorMu} to prove the following bounds on the error terms.

\begin{lem}\label{garvan_lem_circle_Econt}
Let $0 < T < 24$ be an odd integer. Then
\[  \left|\sum_{\substack{0 \leq h < k \leq N\\ (h,k)=1}} e^{-\frac{2\pi
i n h}{k}}\sum_{\substack{t=-\frac{T-1}{2}\\t \neq
0}}^{\frac{T-1}{2}} \int_{-\vartheta_{h,k}'}^{\vartheta_{h,k}''}
E_T^{2,H}(t,l,h,k;z)  d\phi \right| \ll_{r,T} n \log n, \] and, for
$r>2$, we have
\[  \sum_{\substack{0 \leq h < k \leq N\\ (h,k)=1}} e^{-\frac{2\pi
i n h}{k}}\sum_{\substack{t=-\frac{T-1}{2}\\t \neq
0}}^{\frac{T-1}{2}} \int_{-\vartheta_{h,k}'}^{\vartheta_{h,k}''}
E_T^{r,H}(t,l,h,k;z) d\phi  \ll_{r,T} n^{r-1}.\] Furthermore, for
all $r \geq 2$ even, we have
\[  \sum_{\substack{0 \leq h < k \leq N\\ (h,k)=1}} e^{-\frac{2\pi
i n h}{k}}\int_{-\vartheta_{h,k}'}^{\vartheta_{h,k}''}
E_T^{r,\mu}(h,k;z) d\phi  \ll_{r,T} n^{r-1}.\]
\end{lem}

\begin{proof}[Proof of Theorem A]
Proposition \ref{garvan_prop_circle_MuCont}, Proposition
\ref{garvan_prop_circle_Hcont}, and Lemma
\ref{garvan_lem_circle_Econt} prove Theorem A.
\end{proof}
\subsection{Proof of Theorem B}
We now want to identify the leading term in Theorem A. For this
purpose we have to find asymptotic expressions for both the modified
Bessel functions and the integrals over modified Bessel functions
appearing in Theorem A. In fact, it will turn out that the
contributions coming from the integrals over the modified Bessel
functions are smaller than those of the modified Bessel functions
themselves. Thus, we only have to find a precise description of the
former and bound the latter. The key ingredient for both parts is
the following well-known approximation of the Bessel function which
follows from \cite{nist} 10.40(i). As $y \to \infty$, we have
\begin{equation}\label{garvan_eqn_mainterm_bessel} I_\nu(y)=\frac{e^{y}}{\sqrt{2 \pi y}} +
O\left(y^{-\frac{3}{2}}e^y\right). \end{equation}

\begin{prop}\label{garvan_prop_mainterm_bessel}
Let $T<24$ be an odd integer and $r$ an even integer. Then, we have
\begin{equation}\label{garvan_eqn_mainterm_bessel2}
\begin{aligned} &2 \pi \sum_{k \leq \sqrt{n}} \frac{K_k(n)}{k}
\sum_{2a+2b+2c=r} \kappa(a,b,c) (kT)^a
(24n-1)^{-\frac{3}{4}+\frac{a}{2}+c}
I_{-\frac{3}{2}+a+2c}\left(\tfrac{\pi \sqrt{24n-1}}{6k}\right)\\
&=
2\sqrt{3}(-1)^\frac{r}{2}B_{r}\left(\tfrac{1}{2}\right)(24n)^{\frac{r}{2}-1}e^{\pi
\sqrt{\frac{2n}{3}}} +
O_{r,T}\left(n^{\frac{r}{2}-\frac{3}{2}}e^{\pi
\sqrt{\frac{2n}{3}}}\right).\end{aligned}\end{equation}
\end{prop}
\begin{proof}
Equation \eqnref{garvan_eqn_mainterm_bessel} easily implies that the
arguments of the modified Bessel functions play the decisive role
with regards to the asymptotic behavior of
\eqnref{garvan_eqn_mainterm_bessel2}. We see that the leading
contribution must come from $k=1$. All terms coming from $k>1$ are
even exponentially smaller than the error term in the $k=1$
approximation for the modified Bessel function. As a result we can
omit all these terms and place them in the error term. By
\eqnref{garvan_eqn_mainterm_bessel} we obtain
\begin{align*} I_{-\frac{3}{2}+a+2c}\left(\frac{\pi
\sqrt{24n-1}}{6}\right)
&=\frac{\sqrt{3}}{\pi(24n)^\frac{1}{4}}e^{\pi
\sqrt{\frac{2n}{3}}}+O\left(\frac{1}{n^\frac{3}{4}}e^{\pi
\sqrt{\frac{2n}{3}}}\right).\end{align*}

Furthermore, observe that $(24n-1)^{-\frac{3}{4}+\frac{a}{2}+c}$ is
maximized if $2c=r$. In this case we obtain the expression
$(24n-1)^{-\frac{3}{4}+\frac{r}{2}}$. The next smaller term comes
from $2a=2$ and $2c=r-2$, in this case we obtain
$(24n-1)^{-\frac{3}{4}+\frac{r}{2}-\frac{1}{2}}$. This summand and
all others corresponding to the possible choices of $a,b,c$ will
also contribute to the error term with a size smaller than the error
term coming from the approximation of the modified Bessel function
itself. Bounding $\kappa(a,b,c)$ by a constant depending on $r$ and
$T$, we find that \begin{align*} \sum_{2a+2b+2c=r} \kappa(a,b,c)
&T^a (24n-1)^{-\frac{3}{4}+\frac{a}{2}+c}
I_{-\frac{3}{2}+a+2c}\left(\tfrac{\pi \sqrt{24n-1}}{6}\right)\\
&=
\frac{\sqrt{3}}{\pi}\kappa\left(0,0,\tfrac{r}{2}\right)(24n)^{\frac{r}{2}-1}e^{\pi
\sqrt{\frac{2n}{3}}}+O_{r,T}\left(n^{\frac{r}{2}-\frac{3}{2}}e^{\pi
\sqrt{\frac{2n}{3}}}\right).
\end{align*}

Observing that $K_1(n)=1$ and
$\kappa\left(0,0,\tfrac{r}{2}\right)=(-1)^\frac{r}{2}B_{r}\left(\tfrac{1}{2}\right)$,
we conclude the proof of this proposition.
\end{proof}

As a next step we need to find an asymptotic bound for the function
\[ \mathcal{I}_{T;\alpha_{T,t}\left(l,\frac{k}{\gammaGCD}\right),\frac{1}{12}-\frac{\gammaGCD^2}{T^3}\left(\varrho^2+\frac{T^2}{4}-|\varrho|T\right),
-\frac{1}{12},\frac{\varrho}{T}}\left(c,-\tfrac{1}{2}-a-c,k;n\right),
\]
which appears in Theorem A. For simplicity we treat the general
case. Recall the definition \eqnref{garvan_eqn_integral_besselint}.
%
%
%
\begin{prop}\label{garvan_prop_mainterm_prinint}
Let $T>0$ be an odd integer and $r>0$ an even integer and suppose
that $\alpha$, $\beta$, $\delta$, $\varrho$, $c$, $d$, $h$, and $k$
satisfy the usual conditions. Further suppose that $\beta > 0$ and
$d+\tfrac{1}{2} \leq 0$. Then, we have
\[  \mathcal{I}_{T;\alpha,\beta,\delta,\varrho}(c,d,k;n) = O_{r,T}\left(n^{\frac{d}{2}+\frac{1}{4}}e^{\frac{2 \pi \sqrt{\beta2n}}{k}}\right). \]
\end{prop}
\begin{proof}
In order to find the bounds for $\mathcal{I}$, we cannot simply plug
in the approximation for the modified Bessel function (as in
\eqnref{garvan_eqn_mainterm_bessel}) into equation
\eqnref{garvan_eqn_integral_besselint} and then estimate the
resulting integral. This is due to the fact that the approximation
in \eqnref{garvan_eqn_mainterm_bessel} only gives reasonable bounds
for sufficiently large $y$. However, as $y \to 0$, the function
$I_\nu(y)$ decays if $\nu>0$. This is not accounted for in
\eqnref{garvan_eqn_mainterm_bessel}.

For that reason we split the integral for
$\mathcal{I}_{T;\alpha,\beta,\delta,\varrho}(c,d,k;n)$ according to
whether the argument of the modified Bessel function  in \eqnref{garvan_eqn_integral_besselint} is smaller or
bigger than 1. We
abbreviate the corresponding integrals as $\int_{>1}$ and
$\int_{<1}$. Using the approximation for
the modified Bessel function \eqnref{garvan_eqn_mainterm_bessel} we find
\[ \left|\int_{>1}\right|\ll_{r,T} n^{\frac{d}{2}+\frac{1}{4}}e^{\frac{2 \pi
\sqrt{\beta2n}}{k}}.\] For the other case $\int_{<1}$ we use
\cite{nist} 10.30(i), to find that $I_\nu(y)\sim
\frac{1}{\Gamma(\nu+1)}\left(\frac{y}{2}\right)^\nu$, as $y \to 0$.
Using this approximation of the modified Bessel function, one can
easily show that
\[ \left|\int_{<1}\right| \ll_{r,T} n^{-\frac{d+1}{2}}.\]
\end{proof}

\begin{proof}[Proof of Theorem B]
We only show the first assertion. The second one can be dealt with
similarly.

First, we observe that term appearing in the statement of Theorem B is exactly the leading term in
Proposition \ref{garvan_prop_mainterm_bessel}. Hence, it suffices to show that all contributions in Theorem A involving the
functions
\begin{equation}\label{garvan_eqn_mainterm_prinint}\mathcal{I}_{T;\alpha_{T,t}\left(l,\frac{k}{\gammaGCD}\right),\frac{1}{12}-\frac{\gammaGCD^2}{T^3}\left(\varrho^2+\frac{T^2}{4}-|\varrho|T\right),
-\frac{1}{12},\frac{\varrho}{T}}\left(c,-\tfrac{1}{2}-a-c,k;n\right)
\end{equation}
are small. Indeed, looking at the result of Proposition
\ref{garvan_prop_mainterm_prinint}, we see that
\eqnref{garvan_eqn_mainterm_prinint} is bounded by
$n^{-\frac{a+c}{2}}e^{\frac{2 \pi \sqrt{\beta2n}}{k}}$, where
\[ \beta:=\tfrac{1}{12}-\tfrac{\gammaGCD^2}{T^3}\left(\varrho^2+\tfrac{T^2}{4}-|\varrho|T\right).\]
Hence, by Proposition \ref{garvan_prop_mainterm_bessel}, we see that
\eqnref{garvan_eqn_mainterm_prinint} is smaller than the error term
appearing in Proposition \ref{garvan_prop_mainterm_bessel} if we can
show that $\beta<\frac{1}{12}$. This is equivalent to showing that
$\tfrac{\gammaGCD^2}{T^3}\left(\varrho^2+\tfrac{T^2}{4}-|\varrho|T\right)>0.$
To see that this is indeed true, first observe that $|\varrho| \leq
\frac{T-1}{2}$. In this range the function $\varrho \mapsto
\varrho^2-|\varrho|T+\tfrac{T^2}{4}$ attains its minimum at the
boundary $ \varrho = \pm \frac{T-1}{2}$ with value
$\left(\tfrac{T-1}{2}\right)^2-\tfrac{T-1}{2}T+\tfrac{T^2}{4}=\tfrac{1}{4}.$
\end{proof}


\end{document}